\documentclass[lettersize, 11pt]{article}
\usepackage{amsfonts, verbatim, amsmath, color, tikz, subeqnarray}
\usetikzlibrary{arrows.meta}
\usepackage[shortlabels]{enumitem}
\usepackage[title]{appendix}
\usepackage{hyperref}

\newtheorem{theorem}{Theorem}
\newtheorem{lemma}{Lemma}

\newtheorem{proposition}{Proposition}
\newtheorem{remark}{Remark}

\definecolor{darkblue}{rgb}{0.0,0.0,1.0}
\hypersetup{colorlinks,breaklinks,
            linkcolor=darkblue,urlcolor=darkblue,
            anchorcolor=darkblue,citecolor=darkblue}

\def\K {{\cal S}_+}

\def\B{{\mathbb{B}}}
\def\R{{\mathbb{R}}}
\def\Z{{\mathbb{Z}}}

\def\C{{\mathcal C}}

\def\T{{\mathcal T}}
\def\K{{\mathcal K}}
\def\Ve{{\mathcal V}}

\def\Pe{{\mathcal P}}

\def\He{{\mathcal H}}
\def\Oe{{\mathcal O}}

\def\Cupper{{\overline C}}
\def\Clower{{\underline C}}
\def\Vupper{{\overline V}}

\definecolor{myblue}{RGB}{3,70,148}

\def\rred{\textcolor{black}}
\def\rblue{\textcolor{purple}}

\allowdisplaybreaks

\begin{document}

\title{ {\bf Polynomial Time Algorithms and Extended Formulations for 
Unit Commitment Problems}\\[3mm]
\author{\normalsize {\bf Yongpei Guan}$^{\ddag}$, {\bf Kai Pan}$^{\dag}$\footnote{Corresponding author}, and {\bf Kezhuo Zhou}$^{\ddag}$
\\ \\
{\small $^{\ddag}$Department of Industrial and Systems Engineering, University of Florida} \\
{\small Gainesville, Florida 32611, USA. Emails: guan@ise.ufl.edu; zhoukezhuo@ufl.edu} \\
[2mm]
{\small $^{\dag}$Department of Logistics and Maritime Studies, Hong Kong Polytechnic University} \\
{\small Hung Hom, Kowloon, Hong Kong. Email: kai.pan@polyu.edu.hk} \\
}
}

\date{}

\maketitle

\vspace{-0.3in}

\begin{abstract}
\setlength{\baselineskip}{18pt}
Recently increasing penetration of renewable energy generation brings challenges for power system operators to perform efficient power generation daily scheduling, due to the intermittent nature of the renewable generation and discrete decisions of each generation unit. Among all aspects to be considered, a unit commitment polytope is fundamental and embedded in the models at different stages of power system planning and operations. In this paper, we focus on deriving polynomial time algorithms for the unit commitment problems with general convex cost function and piecewise linear cost function respectively. We refine an $\mathcal{O}(T^3)$ time, where $ T $ represents the number of time periods, algorithm for the deterministic single-generator unit commitment problem with general convex cost function and accordingly develop an extended formulation in a higher dimensional space that can provide an integral solution, in which the physical meanings of the decision variables are described. It means the original problem can be solved as a convex program instead of a mixed-integer convex program. Furthermore, for the case in which the cost function is piecewise linear, by exploring the optimality conditions, we derive more efficient algorithms for both deterministic (i.e., $\mathcal{O}(T)$ time) and stochastic (i.e., $\mathcal{O}(N)$ time, where $N$ represents the number of nodes in the stochastic scenario tree) \rred{single-generator} unit commitment problems. We also develop the corresponding extended formulations for both deterministic and stochastic single-generator unit commitment problems that solve the original mixed-integer linear programs as linear programs. Similarly, physical meanings of the decision variables are explored to show the insights of the new modeling approach.

\vspace{0.25in}

\end{abstract}

\newpage
\setcounter{page}{1}
\setlength{\baselineskip}{20pt}
\section{Introduction}
Unit commitment (UC) is fundamental in power system operations. It decides the unit commitment status (online/offline) and power generation amount at each time period for each unit over a finite discrete time horizon, with the objective of minimizing the total cost while satisfying the load (energy demand). Each unit should satisfy associated physical restrictions, such as generation upper/lower limits, ramp-rate limits, and minimum-up/-down time limits. 

Due to its significant importance in power system operations, UC has brought broad attention in academic and industry. In early 1960s, a dynamic programming algorithm was developed in~\cite{lowery1966generating} to formulate and solve the UC problem with single time period, in which the generation amount is discretized and the algorithm itself is not polynomial. Later on, in \cite{snyder1987dynamic}, a more general dynamic programming approximation algorithm (also by discretizing the generation amount) was developed to solve the problem with multiple time periods. Since these algorithms do not run in polynomial time, they are almost intractable. To target large-size problems, other solution approaches such as Lagrangian relaxation (see, e.g., \cite{zhuang1988towards, cheng2000unit}), genetic algorithms (see, e.g., \cite{kazarlis1996genetic, swarup2002unit}), and simulated annealing (see, e.g., \cite{zhuang1990unit, mantawy1998simulated}), have been developed to solve the problem. Detailed reviews of these approaches to solve the UC problem can be found in \cite{padhy2004unit} and \cite{saravanan2013solution}. Among these approaches, the Lagrangian relaxation approach \rred{was} broadly adopted in industry, due to its advantages of decomposing the network constrained UC problem \rred{with multiple generators} into a master problem and a group of subproblems where each subproblem solves an individual UC problem \rred{with a single generator}.

However, the Lagrangian relaxation approach does have \rred{limitations}. For instance, it cannot guarantee to provide an optimal or even a feasible solution at the termination, in particular, when there are transmission network constraints currently faced by most wholesale markets, operated by Independent System Operators (ISOs), in US.  On the other hand, advanced mixed-integer-linear programming (MILP) techniques have been improved significantly during the past decades, and meanwhile MILP in general has advantages in terms of ease of development and maintenance, ability to specify accurate solutions, and exact modeling of complex functionality \cite{nemhauser2013ip}. Thus, \rred{recently, the Lagrangian relaxation has been replaced with MILP approaches 
for power system operations.}
For instance, MILP approaches have been adopted by all ISOs in US (see, e.g., \cite{jabr2012tight, carlson2012miso}) and create more than \$$500$ million annual savings \cite{nemhauser2013ip}. Among different aspects MILP can contribute to the power system operations, one particularly important one is to formulate and solve the transmission network-constrained UC problem \cite{bixby2010mixed}.

The earliest MILP UC formulation was proposed in the 1960s as described in \cite{garver1962power}, and further improvements have been developed recently. For instance, in \cite{de2002price}, an exact and computationally efficient MILP formulation is provided to address the single-generator UC  problem in order to maximize the total profit. \rred{Following that, tighter approximated MILP formulations are provided in \cite{carrion2006computationally, frangioni2009tighter, wu2011tighter,viana2013new}
by approximating the general convex cost function.} In \cite{fu2007fast}, security-constrained UC problems are modeled and solved through the MILP approach for large-scale power systems with multiple generators, \rred{which are further solved in extensive studies by using 
heuristic (including metaheuristic and local search) approaches \cite{roy2013solution, viana2013new, rahman2014metaheuristic, roque2014hybrid}, among others.}
Furthermore, there has been research progress on developing strong formulations for the {UC} problem by exploring its special structure {so as to speed up the algorithm to solve the problem}. For instance, in \cite{lee2004min}, alternating up/down inequalities are proposed to strengthen the minimum-up/-down time polytope of the UC problem. In \cite{rajan2005minimum}, the convex hull of the minimum-up/-down time polytope considering start-up costs is provided, in which additional start-up and shut-down variables are introduced to provide the integral formulation. \rred{In \cite{queyranne2017tight}, a further polyhedral study is provided to consider the bounded up/down times and interval-dependent start-ups.}  Recently, new families of strong valid inequalities are proposed in \cite{ostrowski2012tight}, \cite{damci2015polyhedral}, and \cite{pan2015deteruc} to tighten the ramping polytope of the UC problem.

Besides this, considering the intermittent nature of increasing penetration of renewable generation and generation amount dependency among time periods, recently, a multistage stochastic version of the deterministic UC {was} proposed in~\cite{jiang2016cutting}. In this approach, the extension of the deterministic UC was studied in which a stochastic programming approach (see, e.g., \cite{Shapiro2009SP}) is adopted to address uncertain problem parameters. We refer to the resulting model as the \emph{stochastic UC problem}, which can help make decisions adaptively accommodating uncertainties. \rred{Significant studies have been conducted to solve the stochastic UC problems efficiently, such as advanced decomposition and Lagrangian relaxation techniques~\cite{zheng2013decomposition, papavasiliou2013multiarea}, 
among others.} 
\rred{For notation clarification, in the remaining part of this paper, we denote the single-generator unit commitment problem as single-UC and the multiple-generator unit commitment problem as multi-UC.}

In general, both the deterministic and stochastic UC problems can be eventually formulated as MIPs, with general convex (typically quadratic) cost function (leading to MIQPs) or with piecewise linear cost function approximations (leading to MILPs). 
The perfect cases for solving MILPs are to (1) derive polynomial time algorithms to solve the problems and (2) discover extended formulations in the form of linear programs that can provide integral solutions. 
For the polynomial time algorithms, an efficient algorithm for the deterministic \rred{single-}UC with general convex cost function was studied in~\cite{frangioni2006solving} in which an $\Oe(T^3)$ time algorithm, where $T$ represents the number of time periods, is developed. In this paper, we first refine this algorithm by reducing the computational time to solve the master \rred{single-}UC problem from $\mathcal{O}(T^3)$ time in \cite{frangioni2006solving} to $\mathcal{O}(T^2)$ time, when the economic dispatch problem has been presolved. Meanwhile, based on our developed $\mathcal{O}(T^3)$ time algorithm, we derive an extended formulation in a higher dimensional space. Then, we study the \rred{single-}UC problems with piecewise linear cost functions, which are common in practice. For these cases, we discover the optimality conditions for the generation amounts in optimal solutions at each time period for both the deterministic and stochastic \rred{single-}UC problems. This key observation helps reduce the search space significantly and thus leads to efficient polynomial time algorithms. Accordingly, we develop much faster dynamic programming algorithms that run in $\mathcal{O}(T)$ time to solve the deterministic \rred{single-}UC problem and $\mathcal{O}(N)$ time to solve the stochastic one, where $N$ represents the number of nodes in the scenario tree.
Towards the extended formulation of UC, a recent study for the deterministic \rred{single-}UC problem was provided in \cite{knueven2016generating}, in which an integral polytope for the \rred{single-}UC problem is provided by using the theorem described in~\cite{balas1979disjunctive, balas1998disjunctive}.
\rred{However, it is not stated if this formulation could provide an integral optimal solution for the general convex cost function case.}
In this paper, we provide extended formulations {that can provide integral optimal solutions} based on the innovative dynamic programming algorithms we developed. {Furthermore,} physical meanings of the decision variables in the extended formulations are elaborated. To summarize, the main studies of this paper can be described as follows:
\begin{itemize}
	\item[(1)] In Section \ref{Ndeteruc}, we refine an $\mathcal{O}(T^3)$ time algorithm for the deterministic \rred{single-}UC problem with general convex cost function, which solves the problem in $\mathcal{O}(T^2)$ time when the economic dispatch problem has been presolved. We further derive an extended formulation that can provide an integral optimal solution (which indicates that the original problem can be solved as a convex program instead of a mixed-integer convex program).
	\item[(2)] \rred{In Section \ref{deteruc}, the general convex cost function is approximated by a piecewise linear function. By exploring the optimality conditions for the deterministic \rred{single-}UC problem, we derive a more efficient algorithm that runs in $\Oe(T)$ time and develop the corresponding extended formulation (which solves the original mixed-integer linear program as a linear program).}
	\item[(3)] \rred{In Section \ref{polysuc}, we further consider the uncertainty and formulate the stochastic \rred{single-}UC problem. By exploring the optimality conditions for the corresponding derived multistage stochastic \rred{single-}UC problem with piecewise linear cost function, we derive an efficient algorithm that runs in $\Oe(N)$ time and explore the corresponding extended formulation.}
\end{itemize}

To the best of our knowledge, in this paper, (i) we provide the most efficient polynomial time algorithms to solve the deterministic single-UC problems with both general convex and piecewise linear cost functions; (ii) we provide one of the first extended formulations (with rigorous proofs) that can provide optimal integral solutions for the deterministic single-UC problems with piecewise linear cost functions; (iii) we provide the first studies on the polynomial time algorithm development and extended formulations for the stochastic single-UC problems.

\section{Deterministic Single-UC with General Convex Cost Function} \label{Ndeteruc}
We first introduce the notation and describe the deterministic single-UC problem. We let $T$ be the number of time periods for the whole operational horizon, $L$ ($\ell$) be the minimum-up (-down) time limit, $\Cupper$ ($\Clower$) be the generation upper (lower) bound when the {generator} is online, $\Vupper$ be the start-up/shut-down ramp rate (which is usually between $\Clower$ and $\Cupper$, i.e., $\Clower \leq \Vupper \leq \Cupper$), and $V$ be the ramp-up/-down rate in the stable generation region. In addition, we let binary decision variable {$y_t$} represent the {generator}'s online (i.e., $y_t = 1$) or offline (i.e., $y_t = 0$) status, binary decision variable {$u_t$}  represent whether the {generator} starts up (i.e., $u_t = 1$) or not (i.e., $u_t = 0$), and continuous decision variable {$x_t$} represent the generation amount at time $t$. Moreover, we define continuous variable $SU_t$ to represent the start-up cost at time $t$ (i.e., the generator starts up at time $t$ with $y_{t-1} = 0$ and $y_t = 1$) following the start-up profile and 
continuous variable $SD_t$ to represent the shut-down cost at time $t$ (i.e., the generator shuts down at time $t+1$ with $y_t = 1$ and $y_{t+1} = 0$).
In addition, we let a general convex cost function $f(\cdot)$ denote the fuel cost minus the revenue as a function of its electricity generation amount, online/offline statuses, and electricity price. {Without loss of generality, we} {(1) assume that the cost function $ f(\cdot) $ is strictly convex with respect to the generation amount $ x $; (2)} assume that the {generator} has been offline for $ s_0 $ time periods ($ s_0\geq \ell $) before time $1$. The corresponding deterministic \rred{single-}UC problem can be described as follows:
\begin{subeqnarray} \label{model:Nduc}
	&\min  & \sum_{t=1}^{T} \bigg( SU_t + f_t(x_t,y_t) \bigg) + \sum_{t=L}^{T-1} SD_t \slabel{eqn:Nduc_obj} \\
	&\mbox{s.t.}& \sum_{i = t- L +1}^{t} u_i \leq y_t, \ \forall t \in [L, T]_{\Z}, \slabel{eqn:Np-minup} \\
	&& \sum_{i = t- \ell +1}^{t} u_i \leq 1 - y_{t- \ell}, \ \forall t \in [\ell, T]_{\Z}, \slabel{eqn:Np-mindn} \\
	&& y_t - y_{t-1} - u_t \leq 0, \ \forall t \in [1, T]_{\Z}, \slabel{eqn:Np-udef} \\
	&&  - x_t + \Clower y_t \leq 0, \ \forall t \in [1, T]_{\Z}, \slabel{eqn:Np-lower-bound} \\
	&& x_t - \overline{C} y_t \leq 0, \ \forall t \in [1, T]_{\Z}, \slabel{eqn:Np-upper-bound} \\
	&& x_t - x_{t-1} \leq V y_{t-1} + \Vupper (1 - y_{t-1}), \ \forall t \in [1, T]_{\Z}, \slabel{eqn:Np-ramp-up} \\
	&& x_{t-1} - x_t \leq V y_t + \Vupper (1-y_t), \ \forall t \in [1, T]_{\Z} \slabel{eqn:Np-ramp-down},\\
	&& SU_t \geq SU(t+s_0 -1) \bigg(u_t - \sum_{s=1}^{k-1}y_s \bigg),\ \forall t \in [1, T]_{\Z}, \slabel{eqn:Nsu1} \\
	&& SU_t \geq SU(t-k -1) \bigg(u_t - \sum_{s=k+1}^{t-1}y_s \bigg), \ \forall t \in [L+\ell +1, T]_{\Z}, k \in [L, t-\ell-1]_{\Z},\slabel{eqn:Nsu2} \\
	&& SD_t \geq SD(t-k+1) \bigg(u_k - \sum_{s=k}^{t} (1-y_s) \bigg), \ \forall t \in [L, T-1]_{\Z}, k \in [1, t-L+1]_{\Z},\slabel{eqn:Nsd} \\
	&& y_t, u_t \in \{0, 1\},\  SU_t, SD_t \geq 0,  \ x_0 = y_0=0, \slabel{eqn:Np-nonnegativity}
\end{subeqnarray}
where constraints \eqref{eqn:Np-minup} and \eqref{eqn:Np-mindn} describe the minimum-up and minimum-down time limits, respectively (if the {generator} starts up at time $t-L+1$, it should stay online in the following $L$ consecutive time periods until time $t$; if the {generator} shuts down at time $t-\ell+1$, it should stay offline in the following $\ell$ consecutive time periods until time $t$), constraints \eqref{eqn:Np-udef} describe the logical relationship between $y$ and $u$, constraints \eqref{eqn:Np-lower-bound} and \eqref{eqn:Np-upper-bound} describe the generation lower and upper bounds, and constraints \eqref{eqn:Np-ramp-up} and \eqref{eqn:Np-ramp-down} describe the generation ramp-up and ramp-down rate limits. Constraints \eqref{eqn:Nsu1} describe the start-up cost if the {generator} starts up for the first time, where $ SU(\cdot) $ is a start-up cost function whose variable is the offline time length before starting up. Constraints \eqref{eqn:Nsu2} describe the start-up cost when the {generator} starts up some time later after {a shut-down}. Constraints \eqref{eqn:Nsd} describe the shut-down cost, where $SD(\cdot) $ is a shut-down cost function whose variable is the online time length before shutting down. Typically $SD(\cdot)$ is a constant {number}. In the above formulation, the objective is to minimize the total cost minus the revenue. 
For notation convenience, we define $[a,b]_{\Z}$ with $a < b$ as the set of integer numbers between integers $a$ and $b$, i.e., $\{a, a+1, \cdots, b\}$. 

\subsection{A Refined $\Oe(T^3)$ Time Dynamic Programming Algorithm} \label{subsec:t3-alg}
The polynomial time algorithm for the deterministic single-UC problem with general convex cost function was first developed in \cite{frangioni2006solving}, where an $ \Oe(T^3) $ time dynamic programming algorithm is proposed. This algorithm keeps tracking the ``on'' periods for the {generator} and use a backward dynamic programming {algorithm} to solve the problem. In this algorithm, {each {element in} the state space, i.e., a time{-index} pair $(h,k)$} {with} $k \geq h + L -1$, represents the generator is on during the {time} period $[h,k]_{\Z}$, i.e., the generator is turned on at time $ h $ and turned off at time $ k +1 $. Then the Bellman equation can be written as follows:
\begin{eqnarray}
&& {V_0 (h,k) = C(h,k) + \min_{r \geq k + \ell +1} \Big\{ SU(r-k-1) + V_0(r,q), 0 \Big\},} \nonumber
\end{eqnarray}
for {all possible pairs $ (h,k), (r,q) $ in the state space}. In this equation, $ C(t,k) $ \rblue{represents the optimal cost, equals to generation cost minus revenue, in the interval $[t, k]$}
if the {generator} starts up at time $ t $ and shuts down at time $k + 1$ (online at $k$).
We call this {dynamic programming algorithm as }``part II'' of the algorithm {(solving the problem {by} assuming $C(h,k)$ for all possible pairs $(h,k)$ are given).}
An efficient shortest path algorithm was developed in \cite{frangioni2006solving} to solve this ``part II'' in $ \Oe(T^3) $ {time}.
To speed up the {whole} algorithm, the economic dispatch problem {to calculate $C(h,k)$ for all possible pairs $(h,k)$, i.e.,} for all possible ``on'' intervals, can be precalculated. 
We call this {precalculation of the economic dispatch problem} {as} ``part I'' of this algorithm. In \cite{frangioni2006solving}, an intelligent algorithm was developed to solve {the} ``part I'' problem in $\Oe(T^3)$ time as well.

As compared to \cite{frangioni2006solving}, we propose a more efficient polynomial-time dynamic programming algorithm for ``part II''. We first define {the} optimal value function and then develop the Bellman equations accordingly. The key difference as compared to \cite{frangioni2006solving} is that we use different state spaces and double value functions corresponding to each time period $t$. For instance, we let $ V_\uparrow (t) $ represent the cost from time $ t $ to the end when the generator starts up at time $ t $, as shown in Figure \ref{fig:1st_value}, and $V_\downarrow (t)$ represent the cost from time $t$ to the end when the generator shuts down at time $t +1$ (i.e., $t$ is the last ``on'' period for the current ``on'' interval), as shown in Figure \ref{fig:2nd_value}. Thus, we have the following dynamic programming equations:
\begin{subeqnarray} \label{new_dp1}
	&& V_\uparrow (t) = \min_{  \substack{ k \in [\min \{t+L-1, \\  T-1\}, T-1]_{\Z} }  } \bigg\{ SD(k-t+1) + C(t,k) +  V_\downarrow (k), C(t,T) +  V_\downarrow (T) \bigg\}, \ \forall t \in [1, T]_{\Z},  \slabel{eqn:Ndp1} \\
	&& V_\downarrow (t) = \min_{ k \in [t+\ell+1,T]_{\Z} } \bigg\{ SU(k-t-1) +  V_\uparrow (k), 0 \bigg\}, \ \forall t \in [L, T-\ell -1]_{\Z},  \slabel{eqn:Ndp2} \\
	&& V_\downarrow (t) = 0, \ \forall t \in [T-\ell, T]_{\Z}. \slabel{eqn:Ndp3}
\end{subeqnarray}
Equations \eqref{eqn:Ndp1} indicate that when the {generator} starts up at time $ t $, it can either keep online until time $ k $ with $ k-t +1 \geq L $ and $k \leq T-1$ or keep online throughout all the remaining time periods.
Equations \eqref{eqn:Ndp2} indicate that when the {generator} shuts down at time $ t+1 $, it can either keep offline to the end or starts up again
at time $k$ with $t+\ell+1 \leq k \leq T$. Following the start-up (resp. shut-down) profile, our start-up (resp. shut-down) function can capture the length of offline (resp. online) time before starting up (resp. shutting down). Equations \eqref{eqn:Ndp3} describe that the {generator} cannot start up again after {it shuts} down at $t+1$ with $T-\ell \leq t \leq T$ due to the minimum-down time limit.

\begin{figure}[htb]
	\centering
\begin{tikzpicture}[
	scale=1, 
	wdot/.style = {
      draw,
      fill = white,
      circle,
      inner sep = 0pt,
      minimum size = 4pt
    },
   	bdot/.style = {
      draw,
      fill = black,
      circle,
      inner sep = 0pt,
      minimum size = 4pt
    } ]
\coordinate (O) at (0,0);
\draw[->] (0,0) -- (10,0) coordinate[label = {below:$T$}] (tmax);
\draw[->] (0,0) -- (0,1.5) coordinate[label = {left:$y$}] (ymax);
\draw (0,1) node[wdot, label = {left:$1$}] {};
\draw[thick, draw=red] (0,0) -- (1.5,0);
\draw (0,0) node[wdot, label = {left:$0$}] {};
\draw[thick, draw=red] (4.5,0) -- (7.5,0);
\draw (1.5,0) node[wdot, label = {below:$t$}] {};
\draw (1.5,-0.7) node{$u_t=y_t=1$};

\draw[->, dotted] (1.5,0) -- (1.5,1);
\draw[thick, draw=myblue] (1.5,1) -- (4.5,1);
\draw[->, dotted] (4.5,1) -- (4.5,0.1);

\draw (4.5,0) node[wdot, label = {below:$k$}] {};
\draw (4.5,-0.7) node{$y_k=1$};
\draw (4.5,-1.2) node{\footnotesize $k \geq t+L-1$};

\draw (7.5,0) node[wdot] {};
\draw[->, dotted] (7.5,0) -- (7.5,1);
\draw[thick, draw=myblue] (7.5,1) -- (9.5,1);

\draw[thick] (1.5,1.2) -- (1.5,1.6);
\draw[->] (1.5,1.4) -- (2.5,1.4);
\draw (2,1.65) node{$V_\uparrow (t)$};

\fill[gray!40,nearly transparent] (1.5,2) -- (1.5,0.1) -- (9.5,0.1) -- (9.5,2) -- cycle;
\end{tikzpicture}
\caption{Optimal value function $V_\uparrow (t)$}\label{fig:1st_value}
\end{figure}

\begin{figure}[htb]
\centering
\begin{tikzpicture}[
	scale=1, 
	wdot/.style = {
      draw,
      fill = white,
      circle,
      inner sep = 0pt,
      minimum size = 4pt
    },
   	bdot/.style = {
      draw,
      fill = black,
      circle,
      inner sep = 0pt,
      minimum size = 4pt
    } ]
\coordinate (O) at (0,0);
\draw[->] (0,0) -- (10,0) coordinate[label = {below:$T$}] (tmax);
\draw[->] (0,0) -- (0,1.5) coordinate[label = {left:$y$}] (ymax);
\draw (0,1) node[wdot, label = {left:$1$}] {};
\draw[thick, draw=red] (0,0) -- (1.5,0);
\draw (0,0) node[wdot, label = {left:$0$}] {};
\draw[thick, draw=red] (4.5,0) -- (7.5,0);
\draw (1.5,0) node[wdot] {};

\draw[->, dotted] (1.5,0) -- (1.5,1);
\draw[thick, draw=myblue] (1.5,1) -- (4.5,1);
\draw[->, dotted] (4.5,1) -- (4.5,0.1);

\draw (4.5,0) node[wdot, label = {below:$t$}] {};
\draw (4.5,-0.7) node{$y_t=1$};
\draw (4.5,-1.1) node{$y_{t+1}=0$};

\draw (7.5,0) node[wdot, label = {below:$k$}] {};
\draw (7.5,-0.7) node{$y_k=1$};
\draw (7.5,-1.1) node{$y_{k-1}=0$};
\draw (7.5,-1.5) node{\footnotesize $k \geq t+\ell+1$};
\draw[->, dotted] (7.5,0) -- (7.5,1);
\draw[thick, draw=myblue] (7.5,1) -- (9.5,1);

\draw[thick] (4.5,1.2) -- (4.5,1.6);
\draw[->] (4.5,1.4) -- (5.5,1.4);
\draw (5,1.65) node{$V_\downarrow (t)$};

\fill[gray!40,nearly transparent] (4.5,2) -- (4.5,0.1) -- (9.5,0.1) -- (9.5,2) -- cycle;

\end{tikzpicture}
\caption{Optimal value function $V_\downarrow (t)$}\label{fig:2nd_value}
\end{figure}

As we consider the deterministic \rred{single-}UC problem from times $1$ to $T$ and assume the {generator} has been offline for $ s_0 $ time periods, our goal is to find out the value of the following function:
\begin{eqnarray} \label{new_dp2}
	z = V_\downarrow (-s_0) := \min_{t \in [1,T]_{\Z}} \Big\{SU(s_0 + t -1) +  V_\uparrow (t), 0 \Big\}.
\end{eqnarray}
In order to obtain the optimal objective value and corresponding optimal solution, we calculate $V_\uparrow (t)$ and $V_\downarrow (t)$ for all $t$ and record the optimal candidates for them. To calculate the value of each optimal value function in Bellman equations \eqref{eqn:Ndp1} -- \eqref{eqn:Ndp3} {corresponding to each} $t \leq T$, we search among the candidate solutions for each {possible} $k \leq T$, which takes $ \Oe(T) $ time. Thus, the total time to calculate $V_\downarrow (-s_0) $ is $ \Oe (T^2) $ for the aforementioned ``part II''. The corresponding optimal solution can be obtained by tracing the optimal candidates for each optimal value function starting from $V_\downarrow (-s_0)$, and this takes $ \Oe (T) $ time in total. In summary, our backward induction dynamic programming algorithm for the deterministic \rred{single-}UC problem takes $ \Oe (T^2) $ time for ``part II'' (i.e., if all $ C(t,k) $ are presolved), meaning that our algorithm refines the algorithm in \cite{frangioni2006solving}. More importantly, our algorithm is {beneficial} to derive a better reformulation in the following section.

\subsection{An Extended Formulation for Deterministic Single-UC with \rblue{Piecewise Linear} Convex Cost Function}
In this section, we reformulate the dynamic program in Section \ref{subsec:t3-alg} into a linear program whose dual formulation eventually provides an extended formulation for the deterministic \rred{single-}UC problem.
By incorporating the dynamic {programming} equations (i.e., \eqref{eqn:Ndp1} - \eqref{eqn:Ndp3} and  \eqref{new_dp2}) as constraints, we obtain the following equivalent linear program:
\begin{subeqnarray} \label{model:LP}
	&\max  & z \slabel{eqn:LP_obj}\\
\rred{(\alpha_t)} \hspace{0.1in}	&\mbox{s.t.} & z \leq SU(s_0 + t -1) +  V_\uparrow (t), \ \forall t \in [1,T]_{\Z}, \slabel{eqn:LP1}\\
\rred{(\beta_{tk})} \hspace{0.1in}	&  & V_\uparrow (t) \leq  SD(k-t+1) +  C(t,k) +  V_\downarrow (k), \nonumber \\
	&& \hspace{1.4in}  \forall k \in [\min \{t+L-1,T-1\}, T-1]_{\Z}, \forall t \in [1, T]_{\Z}, \slabel{eqn:LP2}\\
\rred{(\beta_{tk})} \hspace{0.1in}	&& V_\uparrow (t) \leq  C(t,T) +  V_\downarrow (T), \ \forall t \in [1, T]_{\Z}, \slabel{eqn:LP22}\\
\rred{(\gamma_{kt})} \hspace{0.1in}	&& V_\downarrow (t) \leq SU(k-t-1) +  V_\uparrow (k), \ \forall k \in [t+\ell +1,T]_{\Z}, \forall t \in [L, T-\ell -1]_{\Z}, \slabel{eqn:LP3}\\
\rred{(\theta_{t})} \hspace{0.1in}	&& V_\downarrow (t) = 0, \ \forall t \in [T-\ell, T]_{\Z},\slabel{eqn:LP4}\\
	&& z \leq 0, \ V_\downarrow (t) \leq 0, \ \forall t \in [L,T-\ell -1]_{\Z}.\slabel{eqn:LP5}
\end{subeqnarray}
Note here that the optimal value functions in the dynamic program become decision variables in the above formulation. {To} obtain the value $V_\downarrow (-s_0) $ under the dynamic programming framework it is equivalent to maximize variable $ z $ in the linear program above.

Since the above linear program cannot be solved directly as $ C(t,k) $ (the objective value of {the} economic dispatch problem) are unknown, we first show how to obtain the value of $ C(t,k) $ by discussing two cases, i.e., the cases $ k \leq T-1 $ and $ k = T $, respectively.

When $ k \leq T-1 $, we have the following formulation to calculate $C(t,k)$ with $(t,k)$ given:
\begin{subeqnarray} \label{model:ED}
	C(t,k) = & \min & \sum_{s=t}^{k}  \phi_s  \slabel{eqn:ED_obj} \\
\rred{(\lambda_s^-)} \hspace{0.3in}	& \mbox{s.t.}& -x_s \leq -\underline{C}, \ \forall s \in [t,k]_{\Z}, \slabel{eqn:EDlow}\\
\rred{(\lambda_s^+)} \hspace{0.3in}	&& x_s \leq \overline{C}, \ \forall s \in [t,k]_{\Z}, \slabel{eqn:EDupp}\\
\rred{(\mu_t)} \hspace{0.3in}	&& x_t \leq \overline{V}, \slabel{eqn:EDram1}\\
\rred{(\mu_k)} \hspace{0.3in}	&& x_k \leq \overline{V}, \slabel{eqn:EDram2}\\
\rred{(\sigma_s^+)} \hspace{0.3in}	&& x_s - x_{s-1} \leq V, \ \forall s \in [t+1,k]_{\Z}, \slabel{eqn:EDramup}\\
\rred{(\sigma_s^-)} \hspace{0.3in}	&& x_{s-1} - x_s \leq V, \ \forall s \in [t+1,k]_{\Z}, \slabel{eqn:EDramdo}\\
\rred{(\delta_{sj})} \hspace{0.3in}	&& \phi_s \geq a_j^s x_s + b_j, \ \forall s \in [t,k]_{\Z}, j\in  [1,N]_{\Z}. \slabel{eqn:EDlin}
\end{subeqnarray}
{When} $ k = T $, we {have the corresponding formulation by removing} constraint \eqref{eqn:EDram2} as the {generator} is not required to shut down at time $ T+1 $ if it stays online until time $ T $. Note here {that} we assume the generation cost function to be piecewise linear (as indicated by constraints \eqref{eqn:EDlin} with $N$ pieces). {More specifically}, here we use continuous variable $\phi_s$ to represent the cost at time $s$, while $a_j^s$ and $b_j$ are the slope and intercept of the $j$th piece of the cost function at time $s$, respectively.

Next, to incorporate the economic dispatch constraints (e.g., \eqref{eqn:EDlow} - \eqref{eqn:EDlin}) into our proposed linear program \eqref{model:LP}, we take the dual of the economic dispatch model and embed its dual formulation into model \eqref{model:LP}. For instance, for  $ k \leq T-1 $, we have the dual formulation as follows.
\begin{subeqnarray} \label{model:EDd}
	C(t,k) = &\max & \sum_{s=t}^{k} \bigg(\lambda_s^+ \overline{C} - \lambda^-_s \underline{C} \bigg) + \overline{V} (\mu_t+\mu_k) + \sum_{s=t+1}^{k} V \bigg( \sigma_s^+ + \sigma_s^- \bigg) + \sum_{s=t}^{k}\sum_{j=1}^{N} b_j \delta_{sj} \slabel{eqn:EDd_obj} \\
\rred{(q_{tk}^t)} \hspace{0.3in}	&\mbox{s.t.} & \lambda_t^+ - \lambda_t^- + \mu_t - \sigma_{t+1}^- + \sigma_{t+1}^+ - \sum_{j=1}^{N} a_j^t \delta_{tj} = 0, \slabel{eqn:EDd1}\\
\rred{(q_{tk}^k)} \hspace{0.3in}	&& 	\lambda_k^+ - \lambda_k^-  +\mu_k + \sigma_{k}^- - \sigma_{k}^+ - \sum_{j=1}^{N} a_j^k \delta_{kj} = 0, \slabel{eqn:EDd2}\\
\rred{(q_{tk}^s)} \hspace{0.3in}	&& 	\lambda_s^+ - \lambda_s^-  + \sigma_{s}^- - \sigma_{s+1}^-  - \sigma_{s}^+ + \sigma_{s+1}^+ - \sum_{j=1}^{N} a_j^s \delta_{sj} = 0, \ \forall s \in [t+1,k-1]_{\Z}, \slabel{eqn:EDd3}\\
\rred{(w_{tk}^s)} \hspace{0.3in}	&& \sum_{j=1}^{N} \delta_{sj} = 1, \ \forall s \in [t,k]_{\Z}, \slabel{eqn:EDd4}\\
	&& \lambda_s^\pm \leq 0, \ \forall s \in [t,k]_{\Z}, \ \mu_t \leq 0, \ \mu_k \leq 0, \ \sigma_{s}^\pm \leq 0, \ \forall s \in [t+1,k]_{\Z}, \nonumber\\
	&& \delta_{sj}\geq 0, \ \forall j \in [1,N]_{\Z}, s \in [t,k]_{\Z}, \slabel{eqn:EDd5}
\end{subeqnarray}
where $ \lambda_s^- $ and $ \lambda_s^+ $ are dual variables corresponding to constraints \eqref{eqn:EDlow} and \eqref{eqn:EDupp}, respectively, $ \mu_t $ and $ \mu_k $ are the dual variables corresponding to constraint \eqref{eqn:EDram1} and \eqref{eqn:EDram2}, respectively, $\sigma_{s}^+ $ and $ \sigma_{s}^- $ are dual variables corresponding to constraints \eqref{eqn:EDramup} and \eqref{eqn:EDramdo}, respectively, and $ \delta_{sj} $ are dual variables corresponding to constraints \eqref{eqn:EDlin}. \rred{Note that all of these dual variables are correspondingly labeled in the brackets to the left-hand side of \eqref{model:ED} for an easy reference.} For $k = T$, we obtain the corresponding dual formulation by removing the dual variable $ \mu_k $ from model \eqref{model:EDd}. \rred{Thus, in the following part of this section, we refer to \eqref{model:EDd} as the dual formulation for all possible $k$, where $ \mu_k $ will be removed from \eqref{model:EDd} when $k=T$.}
Now we obtain an integrated linear program, as shown in the following, by plugging the dual formulation of {the} economic dispatch problem and {redefining} $C(t,k)$ to be a decision variable in the following model.
\begin{subeqnarray} \label{model:LP_ED}
	& \max  &  z \slabel{eqn:LP_EDobj}\\
	& \mbox{s.t.}  &  \hspace{-0.1in} \eqref{eqn:LP1} - \eqref{eqn:LP5}, \slabel{eqn:LP_ED1}\\
\rred{(p_{tk})} \hspace{-0.2in}	& &  C(t,k) \leq \sum_{s=t}^{k} \bigg( \lambda_s^+ \overline{C} - \lambda^-_s \underline{C} \bigg) + \overline{V} (\mu_t+\mu_k) + \sum_{s=t+1}^{k} V \bigg( \sigma_s^+ + \sigma_s^- \bigg) + \sum_{s=t}^{k}\sum_{j=1}^{N} b_j \delta_{sj}, \nonumber \\
	&& \hspace{1.85in}  \forall k \in [\min \{t+L-1,T-1\}, T-1]_{\Z}, \forall t \in [1, T]_{\Z}, \slabel{eqn:LP_ED2}\\
\rred{(p_{tk})} \hspace{-0.2in}	&& C(t,T) \leq \sum_{s=t}^{T} \bigg( \lambda_s^+ \overline{C} - \lambda^-_s \underline{C} \bigg) + \overline{V} \mu_t + \sum_{s=t+1}^{T} V \bigg( \sigma_s^+ + \sigma_s^- \bigg) + \sum_{s=t}^{T}\sum_{j=1}^{N} b_j \delta_{sj}, \nonumber \\
	&& \hspace{4.3in} \forall t \in [1, T]_{\Z}, \slabel{eqn:LP_ED22}\\
	&& \eqref{eqn:EDd1} - \eqref{eqn:EDd5}, \  \forall t \in [1, T]_{\Z}, \forall k \in [\min \{t+L-1,T\}, T]_{\Z}.\slabel{eqn:LP_ED3}
\end{subeqnarray}
Note here that the right-hand sides of constraints \eqref{eqn:LP_ED2} and \eqref{eqn:LP_ED22} correspond to the objective function \eqref{eqn:EDd_obj} under the case in which $k \leq T-1$ and the case in which $k = T$, respectively.

In the following, we obtain an integral polytope for the original deterministic \rred{single-}UC model \eqref{model:Nduc}. Before that, we take the dual of the above linear program \eqref{model:LP_ED} and obtain the following dual linear program:
\begin{subeqnarray} \label{model:LP_D1}
	&\min  & \sum_{t=1}^{T} SU(s_0 + t -1) \alpha_t + \sum_{t=1}^{T} \sum_{k=t+L-1}^{T-1} SD(k-t+1) \beta_{tk} +\nonumber\\
	&& \ \sum_{t=L}^{T-\ell -1} \sum_{k=t+\ell+1}^{T} SU(k-t-1) \gamma_{tk} + \sum_{tk \in \T \K }  \sum_{s=t}^{k} w_{tk}^s  \slabel{eqn:LP_D1obj}\\
	&\mbox{s.t.}  & \sum_{t=1}^{T} \alpha_t \leq 1, \slabel{eqn:LP_D11}\\
	&& -\alpha_t + \sum_{k= \rblue{\min\{ t+L-1, T\}} }^{T} \beta_{tk} = 0,\ \forall t \in [1, L+\ell]_{\Z}, \slabel{eqn:LP_D12}\\
	&& -\alpha_t + \sum_{k=\rblue{\min\{ t+L-1, T\}}}^{T} \beta_{tk} - \sum_{k=L}^{t-\ell -1} \gamma_{kt} = 0,\ \forall t \in [L+\ell +1, T]_{\Z}, \slabel{eqn:LP_D13}\\
	&& -\sum_{k=1}^{t-L+1}\beta_{kt} + \sum_{k=t+\ell+1}^{T}\gamma_{tk} \leq 0, \ \forall t \in [L, T-\ell -1]_{\Z}, \slabel{eqn:LP_D14}\\
	&& \theta_t - \sum_{k=1}^{t-L+1}\beta_{kt} = 0, \ \forall t \in [T- \ell, T]_{\Z}, \slabel{eqn:LP_D15}\\
	&& p_{tk} - \beta_{tk} =0, \ \forall tk \in \T\K, \slabel{eqn:LP_D16}\\
	&& q_{tk}^s \leq \overline{C}p_{tk}, \ \forall s \in [t,k]_{\Z}, \forall tk \in \T\K, \slabel{eqn:LP_D17}\\
	&& -q_{tk}^s \leq - \underline{C}p_{tk}, \ \forall s \in [t,k]_{\Z}, \forall tk \in \T\K,   \slabel{eqn:LP_D18}\\
	&& q_{tk}^t \leq \overline{V}p_{tk}, \ \forall tk \in \T\K, \slabel{eqn:LP_D19}\\
	&& q_{tk}^k \leq \overline{V}p_{tk}, \ \forall tk \in \T\K, k\leq T-1 \slabel{eqn:LP_D191}\\
	&& q_{tk}^{s-1} - q_{tk}^s \leq V p_{tk}, \  \forall s \in [t+1,k]_{\Z}, \forall tk \in \T\K, \slabel{eqn:LP_D110}\\
	&& q_{tk}^{s} - q_{tk}^{s-1} \leq V p_{tk}, \  \forall s \in [t+1,k]_{\Z}, \forall tk \in \T\K, \slabel{eqn:LP_D111}\\
	&& w_{tk}^s - a_j^s q_{tk}^{s} \geq b_j p_{tk}, \ \forall s \in [t,k]_{\Z}, j \in  [1,N]_{\Z}, \forall tk \in \T\K, \slabel{eqn:LP_D112}\\
	&& \alpha, \beta, \gamma,p \geq 0,
\end{subeqnarray}
where $ \T\K $ represents the set of all possible combinations of each $t \in [1, T]_{\Z}$ and each $k \in [\min \{t+L-1,T\}, T]_{\Z}$ to construct a time interval $[t,k]_{\Z}$. In the above dual formulation, dual variables $ \alpha, \beta, \gamma$, and $\theta $ \rred{(labeled in the brackets to the left-hand side of \eqref{model:LP})} correspond to constraints \eqref{eqn:LP1} -- \eqref{eqn:LP4}, respectively, and dual variables $ p,q$, and $w$ \rred{(labeled in the brackets to the left-hand sides of \eqref{model:EDd} and \eqref{model:LP_ED})} correspond to constraints \eqref{eqn:LP_ED2} -- \eqref{eqn:LP_ED3} for each $ tk \in \T\K $, respectively.


After replacing $ p $ with $ \beta $ (due to \eqref{eqn:LP_D16}) and combining \eqref{eqn:LP_D12} and \eqref{eqn:LP_D13} in the dual formulation \eqref{model:LP_D1}, we obtain the following cleaner model:
\begin{subeqnarray} \label{model:LP_D2}
	&\min  & \sum_{t=1}^{T} SU(s_0 + t -1) \alpha_t + \sum_{t=1}^{T} \sum_{k=t+L-1}^{T-1} SD(k-t+1) \beta_{tk} +\nonumber\\
	&& \ \sum_{t=L}^{T-\ell -1} \sum_{k=t+\ell+1}^{T} SU(k-t-1) \gamma_{tk} + \sum_{tk \in \T \K }  \sum_{s=t}^{k} \rblue{w_{tk}^s}  \slabel{eqn:LP_D2obj}\\
	&\mbox{s.t.}  & \sum_{t=1}^{T} \alpha_t \leq 1, \slabel{eqn:LP_D21}\\
	&& -\alpha_t + \sum_{k=\rblue{\min\{ t+L-1, T\}}}^{T} \beta_{tk} - \sum_{k=L}^{t-\ell -1} \gamma_{kt} = 0,\ \forall t \in [1, T]_{\Z}, \slabel{eqn:LP_D22}\\
	&& -\sum_{k=1}^{t-L+1}\beta_{kt} + \sum_{k=t+\ell+1}^{T}\gamma_{tk} \leq 0, \ \forall t \in [L, T-\ell -1]_{\Z}, \slabel{eqn:LP_D23}\\
	&& \theta_t - \sum_{k=1}^{t-L+1}\beta_{kt} = 0, \ \forall t \in [T- \ell, T]_{\Z}, \slabel{eqn:LP_D24}\\
	&& \underline{C}\beta_{tk} \leq q_{tk}^s \leq \overline{C}\beta_{tk}, \ \forall s \in [t,k]_{\Z}, \forall tk \in \T\K, \slabel{eqn:LP_D25}\\
	&& q_{tk}^t \leq \overline{V}\beta_{tk}, \ \forall tk \in \T\K, \slabel{eqn:LP_D26}\\
	&& q_{tk}^k \leq \overline{V}\beta_{tk}, \ \forall tk \in \T\K, k \leq T-1 \slabel{eqn:LP_D261}\\
	&& q_{tk}^{s-1} - q_{tk}^s \leq V \beta_{tk}, \  \forall s \in [t+1,k]_{\Z}, \forall tk \in \T\K, \slabel{eqn:LP_D27}\\
	&& q_{tk}^{s} - q_{tk}^{s-1} \leq V \beta_{tk}, \  \forall s \in [t+1,k]_{\Z}, \forall tk \in \T\K, \slabel{eqn:LP_D28}\\
	&& \rblue{w_{tk}^s - a_j^s q_{tk}^{s} \geq b_j \beta_{tk}, \ \forall s \in [t,k]_{\Z}, j \in  [1,N]_{\Z}, \forall tk \in \T\K;} \ \alpha, \beta, \gamma \geq 0. \slabel{eqn:LP_D29}
\end{subeqnarray}

Now, we show that the polytope \eqref{eqn:LP_D21} --  \eqref{eqn:LP_D29} is an integral polytope in the following theorem, which indicates that the extreme points of the polytope are integral.

\begin{theorem} \label{lemma:binary solution_new}
The extreme points of the polytope \eqref{eqn:LP_D21} -- \eqref{eqn:LP_D29} are binary with respect to decision variables $ \alpha, \beta, \gamma$ and $\theta$.
\end{theorem}
\begin{proof}
To prove Theorem \ref{lemma:binary solution_new}, \rblue{since} model \eqref{model:LP_D2} has a linear objective function so that the optimal solution lies in the extreme point set and we only need to show that for any arbitrary linear objective function, the optimal solutions of model \eqref{model:LP_D2} are binary with respect to $ \alpha, \beta, \gamma$, and $\theta$. To that end, without loss of generality, we make the following assumption without modifying the polytope \eqref{eqn:LP_D21} -- \eqref{eqn:LP_D29}.
	\begin{itemize}[nolistsep]
		\item ${w_{tk}^s}  = a_{tk}^{s} q_{tk}^{s} + b_{tk}^{s} \beta_{tk} $, where $ a_{tk}^{s}  $ and $ b_{tk}^{s} $ are  parameters;
		\item $ V_\downarrow (t) = E_t, \ \forall t \in [T-\ell, T]_{\Z}, $ where $ E_t $ is a parameter.
	\end{itemize}
	Based on the above assumptions, we can obtain model \eqref{model:LP_D2} with the same set of constraints \eqref{eqn:LP_D21} -- \eqref{eqn:LP_D29} \rblue{without the term included in \eqref{eqn:LP_D112}} but with the following objective function:
	\begin{eqnarray}
		&\min & \sum_{t=1}^{T} SU(s_0 + t -1) \alpha_t  + \sum_{t=1}^{T} \sum_{k=t+L-1}^{T-1} SD(k-t+1) \beta_{tk} + \sum_{t=L}^{T-\ell -1} \sum_{k=t+\ell+1}^{T} SU(k-t-1) \gamma_{tk}\nonumber\\ 
		&&+ \sum_{t=T-\ell}^{T}E_t \theta_t + \sum_{tk \in \T \K }  \sum_{s=t}^{k} \bigg( a_{tk}^{s} q_{tk}^{s} + b_{tk}^{s} \beta_{tk} \bigg), \slabel{obj_pf}
	\end{eqnarray}
where the coefficient of each variable can be an arbitrary number.

	For notation brevity, we denote \eqref{obj_pf} as $\min c^{\top} (\alpha,\beta,\gamma,\theta,q) $ where $ c $ is the column vector including all {coefficients} in \eqref{obj_pf}. Now we prove {that} for any value of $ c $, we can provide an optimal solution that is integral with respect to $\alpha,\beta,\gamma,\theta  $ to the linear program with \eqref{obj_pf} as the objective function and \eqref{eqn:LP_D21} -- \eqref{eqn:LP_D29} as constraints, which means that Theorem \ref{lemma:binary solution_new} holds.
	
Considering the assumptions we make, we can obtain an optimal solution with the dynamic programming algorithm \eqref{new_dp1} -- \eqref{new_dp2} for any given $c$. Based on this optimal solution, which indicates the online/offline status and generation amount of the {generator} at each time period, we construct a solution $(\alpha^*,\beta^*,\gamma^*,\theta^*,q^*)$ as follows:
\begin{itemize}[noitemsep]
\item[1)] $ \alpha^*_{t} =1$ if the {generator} starts up for the first time at time $ t $ and $ \alpha^*_{t} =0 $ otherwise.
\item[2)] $ \beta^*_{tk} =1 $ if the {generator} starts up at time $ t $ and shuts down at time $ k + 1 $ and $ \beta^*_{tk} =0 $ otherwise.
\item[3)] $ \gamma^*_{tk} =1 $ if the {generator} shuts down at time $ t + 1 $ and starts up at time $ k $ and $ \gamma^*_{tk} =0 $ otherwise.
\item[4)] $ \theta^*_{t} =1 $ if the {generator} shuts down at time $t + 1$ and stays offline to the end and $ \theta^*_{t} =0 $ otherwise.
\item[5)] $q^{s*}_{tk}$ takes the value of optimal generation output for each $s \in [t,k]_{\Z}$ if the {generator} starts up at time $ t $ and shuts down at time $ k +1 $ and $q^{s*}_{tk} = 0$ otherwise.
\end{itemize}
In the following, we show $(\alpha^*,\beta^*,\gamma^*,\theta^*,q^*)$ is an optimal solution of model \eqref{model:LP_D2} with objective function \eqref{obj_pf} and constraints \eqref{eqn:LP_D21} -- \eqref{eqn:LP_D29}.

We first verify the feasibility. Since $ \alpha^*_t =1 $ for at most one period $t \in [1,T]_{\Z}$, constraint \eqref{eqn:LP_D21} is satisfied.
For each $ t \in [1,T]_{\Z} $ in constraints \eqref{eqn:LP_D22}, we consider the following two possible cases: 
\begin{itemize}[nolistsep, leftmargin=*]
\item[1)] if $ \beta^*_{tk} = 0 $ for all $k \in [t+L-1,T]_{\Z}$, by definition, $ \alpha^*_t=0 $ and $ \gamma^*_{kt} = 0 $ for all possible $k \in [L,t-\ell-1]_{\Z}$ if $t \geq L+\ell+1$.
\item[2)] otherwise $ \beta^*_{tk} = 1 $ for some $k \in [t+L-1,T]_{\Z}$, then if the {generator} starts up at time $ t $ for the first time, we have $ \alpha^*_t =1 $ and $ \gamma^*_{kt} = 0 $ for all possible $k \in [L,t-\ell-1]_{\Z}$ if $t \geq L+\ell+1$; else $ \alpha^*_t =0 $ and there exists exactly one $k \in [L,t-\ell-1]_{\Z}$ such that $ \gamma^*_{kt} = 1 $.
\end{itemize}
For both cases, constraints \eqref{eqn:LP_D22} are satisfied.
For each $ t \in [L, T-\ell -1]_{\Z}$ in constraints \eqref{eqn:LP_D23}, we consider the following two possible cases:
\begin{itemize}[nolistsep, leftmargin=*]
\item[1)] if $ \beta^*_{kt} = 0$ for all $k \in [1, t-L+1]_{\Z}$, then $ \gamma_{tk}^*  =0 $ for all $k \in [t+\ell+1, T]_{\Z}$.
\item[2)] otherwise $ \beta^*_{kt} = 1 $ for some $k \in [1, t-L+1]_{\Z}$, meaning that the {generator} shuts down at $t+1$, then it either start up again after satisfying the minimum-down time limit, which indicates that $  \sum_{k=t+\ell+1}^{T}\gamma_{tk}^* =1 $, or stay offline to the end, which indicates $  \sum_{k=t+\ell+1}^{T}\gamma_{tk}^* =0 $.
\end{itemize}
For both cases, constraints \eqref{eqn:LP_D23} are satisfied.
For each $ t \in [T- \ell, T]_{\Z}$ in constraints \eqref{eqn:LP_D24}, we consider the following two possible cases:
\begin{itemize}[nolistsep, leftmargin=*]
\item[1)] if $ \beta^*_{kt} = 1 $ for some $k \in [1,t-L+1]_{\Z}$, then the {generator} cannot start up again after shutting down at $t+1$ due to the minimum-down time limit, it follows that $ \theta_t^* = 1 $.
\item[2)] otherwise $ \beta^*_{kt} = 0$ for all $k \in [1,t-L+1]_{\Z}$, then $ \theta_t^* = 0$.
\end{itemize}
For both cases, constraints \eqref{eqn:LP_D24} are satisfied.
For constraints \eqref{eqn:LP_D25} -- \eqref{eqn:LP_D28}, they are immediately satisfied by the construction of our solution and the definition of the economic dispatch problem. Also, \eqref{eqn:LP_D29} is satisfied obviously.
	
We then verify the optimality. We claim that the objective function \eqref{obj_pf} under the constructed solution equals to the objective value of the dynamic programming algorithm \eqref{new_dp1} -- \eqref{new_dp2} {as follows}.
\begin{subeqnarray}
	&&\sum_{t=1}^{T} SU(s_0 + t -1) \alpha_t^* +  \sum_{t=1}^{T} \sum_{k=t+L-1}^{T-1} SD(k-t+1) \beta_{tk}^* + \sum_{t=L}^{T-\ell -1} \sum_{k=t+\ell+1}^{T} SU(k-t-1) \gamma_{tk}^* \nonumber \\
	&& \ + \sum_{t=T-\ell}^{T}E_t \theta_t^* + \sum_{tk \in \T \K }  \sum_{s=t}^{k} \bigg(a_{tk}^{s} q_{tk}^{s*} + b_{tk}^{s} \beta_{tk}^* \bigg)  \nonumber \\
	& = & SU(s_0 + t_1 -1)  +  \sum_{tk: \beta_{tk}^*=1, k\leq T-1} SD(k-t+1) + \sum_{tk: \gamma_{tk}^*=1} SU(k-t-1)  + E_{t_2} \nonumber \\
	&& \  + \sum_{tk: \beta_{tk}^* =1 }  \sum_{s=t}^{k} \bigg( a_{tk}^{s} q_{tk}^{s*} + b_{tk}^{s} \bigg) \slabel{pf1} \\
	& = & SU(s_0 + t_1 -1)  +  \sum_{tk: \beta_{tk}^*=1, k\leq T-1} SD(k-t+1) + \sum_{tk: \gamma_{tk}^*=1} SU(k-t-1)  + E_{t_2} \nonumber \\
	&& \ + \sum_{tk: \beta_{tk}^* =1 }  C(t,k) \slabel{pf2}\\
	& = & V_\downarrow (-s_0), \slabel{pf3}
\end{subeqnarray}
where $t_1$ and $t_2$ in \eqref{pf1} represent $ \alpha_{t_1}^* = 1$ and $ \theta_{t_2}^* = 1$, respectively, \eqref{pf1} and \eqref{pf2} hold due to our assumption at the beginning of the proof, and \eqref{pf3} holds because $(\alpha^*,\beta^*,\gamma^*,\theta^*,q^*)$ is constructed based on the optimal solution of the dynamic programming algorithm \eqref{new_dp1} -- \eqref{new_dp2} {and is actually the expansion of the objective function in \eqref{new_dp2}}.
By the Strong Duality Theorem, the constructed solution $(\alpha^*,\beta^*,\gamma^*,\theta^*,q^*) $ is an optimal solution for model \eqref{model:LP_D2}.

From the above analysis, we notice that  $(\alpha^*,\beta^*,\gamma^*,\theta^*,q^*) $ is binary with respect to $ \alpha, \beta, \gamma$ and $\theta$ and optimal for the dual program for all possible cost coefficient $c$. Thus, we have proved our claim.
\end{proof}

\begin{proposition} \label{cor:binary solution_new}
\rblue{There exists an} optimal solution {to the dual formulation} \eqref{model:LP_D2} binary with respect to decision variables $ \alpha, \beta, \gamma$ and $ \theta $.
\end{proposition}
\begin{proof}
\rblue{The conclusion directly follows from Theorem \ref{lemma:binary solution_new} and linear objective function \eqref{eqn:LP_D2obj}.}
\end{proof}

From Theorem \ref{lemma:binary solution_new}, we can further observe that this formulation itself has specific physical meanings. In particular, we notice that the optimal solution $ \alpha^*_{t} $ in the dual program represents whether at time $ t $ the {generator} starts up for the first time or not. If yes, then $ \alpha^*_{t} =1 $; otherwise, $ \alpha^*_{t} =0 $. Similarly, if the {generator} starts up from time $ t $ and shuts down at time $ k+1 $, $ \beta^*_{tk} =1 $; otherwise, $ \beta^*_{tk} =0 $. If the {generator} shuts down at time $ t+1 $ and starts up again at time $ k $, $ \gamma^*_{tk} =1 $; otherwise, $ \gamma^*_{tk} =0 $. If the {generator} shuts down at time $ t+1 $ and stays offline to the end, $ \theta^*_t =1 $; otherwise, $ \theta^*_t =0 $. That is, this theorem not only provides an integral polytope for the deterministic \rred{single-}UC problem but also provides an insight to formulate the problem in a different way by linking the original space of the deterministic \rred{single-}UC problem with the derived reformulation \eqref{model:LP_D2}. It follows that we have the following proposition.

\begin{proposition} \label{thm: opt solution_new}
	If  $(\alpha^*,\beta^*,\gamma^*,\theta^*,q^*) $  is an optimal solution to {the} dual program \eqref{model:LP_D2}, then
	\begin{eqnarray}
		&& x_s^* = \sum_{tk \in \T\K, t\leq s \leq k} q^{s*}_{tk}, \ y_s^* = \sum_{tk \in \T\K, t\leq s \leq k} \beta^{*}_{tk},\  u_s^* = \alpha_s^* + \sum_{tk \in \T\K, k= s} \gamma^{*}_{tk}, \ \forall s \in [1,T]_{\Z} \slabel{eqn:opt solution_new} 
	\end{eqnarray}
	is an optimal solution to the deterministic \rred{single-}UC problem \eqref{model:Nduc}.
\end{proposition}
\begin{proof}
	From {Proposition} \ref{cor:binary solution_new}
	and  constraints \eqref{eqn:LP_D25} -- \eqref{eqn:LP_D28},  we can easily conclude that both $ y* $ and $ u^* $ are binary and $(x^*, y^*, u^*)$ satisfies constraints  \eqref{eqn:Np-minup} -- \eqref{eqn:Np-ramp-down}. That is, $(x^*, y^*, u^*)$ is feasible to the deterministic \rred{single-}UC problem. Meanwhile, through plugging $(x^*, y^*, u^*)$ into \eqref{eqn:Nduc_obj}, we can observe that 
	\begin{eqnarray}
	&&	\sum_{t=1}^{T} \bigg( SU_t + f_t(x_t^*, y_t^*) \bigg) + \sum_{t=L}^{T-1} SD_t  \nonumber \\ 
	& = & \sum_{t=1}^{T} SU(s_0 + t -1) \alpha^*_t +  \sum_{tk: \beta_{tk}^*=1, k\leq T-1} SD(k-t+1) \nonumber \\
	&& + \sum_{t=L}^{T-\ell -1} \sum_{k=t+\ell+1}^{T} SU(k-t-1) \gamma^*_{tk} + \sum_{tk \in \T \K }  \sum_{s=t}^{k} { w_{tk}^{s*}} \nonumber
	\end{eqnarray}
	Thus, $(x^*, y^*, u^*)$ is optimal to the deterministic \rred{single-}UC problem.
\end{proof}

Now, based on Proposition \ref{cor:binary solution_new} 
and Proposition \ref{thm: opt solution_new}, we can construct the extended formulation of the deterministic \rred{single-}UC problem 
\rblue{by adding}
equations \eqref{eqn:opt solution_new}  to represent the {relationship} between {the} original decisions and dual decision variables.
\begin{theorem}\label{thm:extend_new}
	The extended formulation of the deterministic \rred{single-}UC problem with a general \rblue{piecewise linear} convex cost function can be written as follows:
	\begin{eqnarray} 
		&\min  & 
		\rblue{\eqref{eqn:LP_D2obj}} \nonumber \\
		&\mbox{s.t.}  & x_s = \sum_{tk \in \T\K, t\leq s \leq k} q^{s}_{tk}, \ y_s = \sum_{tk \in \T\K, t\leq s \leq k} \beta_{tk}, \ u_s = \alpha_s + \sum_{tk \in \T\K, k=s} \gamma_{tk},  \ \forall s \in [1,T]_{\Z}, \nonumber \\
		&& 
		 \eqref{eqn:LP_D21} - \eqref{eqn:LP_D29}, 
	\end{eqnarray}
	and if $ (x^*,y^*,u^*,\alpha^*,\beta^*,\gamma^*,\theta^*,q^*) $ is an optimal solution to the extended formulation, then  $ (x^*,y^*,u^*) $ is an optimal solution to the deterministic \rred{single-}UC problem \eqref{model:Nduc}.
\end{theorem}

\begin{remark}
From the reformulation described above, we can observe that our extended formulation not only describes an integral polytope with respect to variables $(y,u,\alpha,\beta,\gamma,\theta)$ due to Theorem \ref{lemma:binary solution_new}, 
but also provides an optimal objective converging to that of the deterministic single-UC problem \eqref{model:Nduc} when $f(\cdot)$ is a general convex function due to the compactness of the feasible region and bounded objective value for the single-UC problem and Theorem \ref{thm:extend_new}.
\end{remark}

\section{Deterministic \rred{Single-}UC with Piecewise Linear Cost Function} \label{deteruc}

In practice, the general convex cost function is usually approximated by a piecewise linear function.
In this section, we propose a more efficient dynamic programming algorithm {for this type of problem}. To explore the property more conveniently, we consider simplified start-up/shut-down cost {first}. That is, we consider {the} objective function in the following way:
\begin{eqnarray}
\min  & \sum_{t=1}^{T} \Big( \overline{U}u_t + \underline{U}(y_{t-1} - y_t + u_t) + f_t(x_t,y_t) \Big), \label{eqn:org_model_no_su_sd}
\end{eqnarray}
{where $\overline{U}$ and $\underline{U}$ represents the time-invariant start-up and shut-down costs, respectively.} Thus, here we have constraints \eqref{eqn:Np-minup} - \eqref{eqn:Np-ramp-down} plus \eqref{eqn:Np-nonnegativity} and decision variables $ x,y,u $.
We first explore the optimality condition of this problem and develop a new dynamic programming algorithm to solve the deterministic \rred{single-}UC problem. In addition, {an} extended formulation in a higher dimensional space is derived from our proposed algorithm.

\subsection{An Optimality Condition}
{We first derive a property for the extreme points of conv($\mathcal{D}$), where}
$\mathcal{D} = \{(x,y,u) \in \R^T \times \B^{2T}: \eqref{eqn:Np-minup} - \eqref{eqn:Np-ramp-down} \}$ {and conv($\mathcal{D}$) represents the convex hull description of $\mathcal{D}$,} {before exploring the optimality condition. We let}
 $\alpha_1 = \max \{n \in [1,T]_{\Z}: \Clower + nV \leq \Cupper\}$, $\alpha_2 = \max \{n \in [1,T]_{\Z}: \Vupper + nV \leq \Cupper\}$, and 
 \begin{eqnarray}
\mathcal{Q} = \{0, \ (\Clower+nV)_{n=0}^{\alpha_1}, \ (\Vupper+nV)_{n=0}^{\alpha_2}, \ (\Cupper-nV)_{n=0}^{\alpha_1}\}. \slabel{star}
 \end{eqnarray}
Note {here} that $\alpha_2 \leq \alpha_1 \leq T$ because $\Vupper \geq \Clower$. 

\begin{proposition} \label{prop:duc_ext_point}
For any extreme point $(\bar{x}, \bar{y}, \bar{u})$ of conv($\mathcal{D}$), $\bar{x}_t \in \mathcal{Q}$ for all $t \in [1,T]_{\Z}$.
\end{proposition}
\begin{proof}
{We prove this claim by a contradiction method.}
Suppose that there exists some $t \in [1,T]_{\Z}$ such that $\bar{x}_t \notin \mathcal{Q}$ for an extreme point $(\bar{x}, \bar{y}, \bar{u})$ of conv($\mathcal{D}$), i.e., $\bar{x}_t \in (\Clower, \Cupper) \setminus \mathcal{Q}$. In the following we construct two feasible points of conv($\mathcal{D}$) to {show that} $(\bar{x}, \bar{y}, \bar{u})$ is a convex combination of these two points, leading to the contradiction. If $t \geq 2$ and $|\bar{x}_t - \bar{x}_{t-1}| = V$, we let $s_1 \leq t-1$ be the smallest index such that $|\bar{x}_{s_1+1} - \bar{x}_{s_1}| = |\bar{x}_{s_1+2} - \bar{x}_{s_1+1}| = \cdots = |\bar{x}_t - \bar{x}_{t-1}| = V$; otherwise, we let $s_1 = t$. If $t \leq T-1$ and $|\bar{x}_{t+1} - \bar{x}_{t}| = V$, we let $s_2 \geq t+1$ be the largest index such that $|\bar{x}_{t+1} - \bar{x}_{t}| = |\bar{x}_{t+2} - \bar{x}_{t+1}| = \cdots = |\bar{x}_{s_2} - \bar{x}_{s_2-1}| = V$; otherwise, we let $s_2 = t$. We construct two points $(\bar{x}^1, \bar{y}, \bar{u})$ and $(\bar{x}^2, \bar{y}, \bar{u})$ such that $\bar{x}^1_r = \bar{x}_r + \epsilon$ for $r \in [s_1,s_2]_{\Z}$, $\bar{x}^2_r = \bar{x}_r - \epsilon$ for $r \in [s_1,s_2]_{\Z}$, and $\bar{x}^1_r = \bar{x}^2_r = \bar{x}_r$ for $r \notin [s_1,s_2]_{\Z}$, where $\epsilon$ is an arbitrarily small positive number.

Now we show these two points constructed are feasible for conv($\mathcal{D}$) by considering the following three possible cases.
\begin{enumerate}[(1)]
\item If $\bar{x}_{s_1-1} = 0$, then we have $\Clower < \bar{x}_{s_1} < \Vupper$. Otherwise, {we have the following two possible cases.}
	\begin{enumerate}[nolistsep]
	\item If $\bar{x}_{s_1} = \Clower$, it follows that $\bar{x}_t = \Clower + kV$ for some $k \in [0, t-s_1]_{\Z}$ due to $|\bar{x}_{s_1+1} - \bar{x}_{s_1}| = |\bar{x}_{s_1+2} - \bar{x}_{s_1+1}| = \cdots = |\bar{x}_t - \bar{x}_{t-1}| = V$. Since $\bar{x}_t \notin \mathcal{Q}$, it further follows that $k \geq \alpha_1+1 = \min \{T, \lfloor (\Cupper-\Clower)/V \rfloor \}+1$, which contradicts to the fact that $k \leq T$ and $\bar{x}_t < \Cupper$.
	\item If $\bar{x}_{s_1} = \Vupper$, it follows that $\bar{x}_t = \Vupper + kV$ for some $k \in [0, t-s_1]_{\Z}$ due to $|\bar{x}_{s_1+1} - \bar{x}_{s_1}| = |\bar{x}_{s_1+2} - \bar{x}_{s_1+1}| = \cdots = |\bar{x}_t - \bar{x}_{t-1}| = V$. Since $\bar{x}_t \notin \mathcal{Q}$, it further follows that $k \geq \alpha_2+1 = \min \{T, \lfloor (\Cupper-\Vupper)/V \rfloor \}+1$, which contradicts to the fact that $k \leq T$ and $\bar{x}_t < \Cupper$. 
	\end{enumerate}
Similarly, if $\bar{x}_{s_2+1} = 0$, then we have $\Clower < \bar{x}_{s_2} < \Vupper$. Therefore, in either case, it is feasible to increase or decrease $\bar{x}_{s_1}$ and $\bar{x}_{s_2}$ by $\epsilon$.
\item If $\bar{x}_{s_1-1} > 0$, then we have $\Clower < \bar{x}_{s_1} < \Cupper$. Otherwise, {we have the following two possible cases.}
	\begin{enumerate}[nolistsep]
	\item If $\bar{x}_{s_1} = \Clower$, we can similarly show the contradiction as above.
	\item If $\bar{x}_{s_1} = \Cupper$, it follows that $\bar{x}_t = \Cupper - kV$ for some $k \in [0, t-s_1]_{\Z}$ due to $|\bar{x}_{s_1+1} - \bar{x}_{s_1}| = |\bar{x}_{s_1+2} - \bar{x}_{s_1+1}| = \cdots = |\bar{x}_t - \bar{x}_{t-1}| = V$. Since $\bar{x}_t \notin \mathcal{Q}$, it further follows that $k \geq \alpha_1+1 = \min \{T, \lfloor (\Cupper-\Clower)/V \rfloor \}+1$, which contradicts to the fact that $k \leq T$ and $\bar{x}_t > \Clower$.
	\end{enumerate}
Similarly, if $\bar{x}_{s_2+1} > 0$, then we have $\Clower < \bar{x}_{s_2} < \Cupper$. Therefore, in either case, it is feasible to increase or decrease $\bar{x}_{s_1}$ and $\bar{x}_{s_2}$ by $\epsilon$ since $|\bar{x}_{s_1} - \bar{x}_{s_1-1}| < V$ and $|\bar{x}_{s_2+1} - \bar{x}_{s_2}| < V$ by definition.
\item If {$s_1-1, s_2+1 \notin [1,T]_{\Z}$}, i.e., $s_1 = 1$ or $s_2 = T$, then similarly we can follow the arguments above to show that $\Clower < \bar{x}_{s_1} < \Cupper$ and $\Clower < \bar{x}_{s_2} < \Cupper$. It follows that it is feasible to increase or decrease $\bar{x}_{s_1}$ and $\bar{x}_{s_2}$ by $\epsilon$.
\end{enumerate}

In summary, we show that in all cases it is feasible to increase or decrease $\bar{x}_{s_1}$ and $\bar{x}_{s_2}$ by $\epsilon$ and thus feasible to increase or decrease $\bar{x}_r$ by $\epsilon$ for all $r \in [s_1,s_2]_{\Z}$. It follows that both $(\bar{x}^1, \bar{y}, \bar{u})$ and $(\bar{x}^2, \bar{y}, \bar{u})$ are feasible points of conv($\mathcal{D}$), and $(\bar{x}, \bar{y}, \bar{u}) = \frac{1}{2}(\bar{x}^1, \bar{y}, \bar{u}) + \frac{1}{2}(\bar{x}^2, \bar{y}, \bar{u})$. Therefore, $(\bar{x}, \bar{y}, \bar{u})$ is not an extreme point of conv($\mathcal{D}$), which is a contradiction.
\end{proof}

Now, we begin to characterize the optimality condition for {the deterministic \rred{single-}UC} \eqref{eqn:org_model_no_su_sd}. Generally ${f_t(x_t,y_t)} = a x_t^2 + b x_t + c y_t - q_t x_t$, where $(a,b,c)$ is determined by the generator's physical characteristics, and it is often approximated by a $K-$piece piecewise linear function $\varphi_t = f_t(x_t,y_t) \geq \mu_k^t x_t + \nu_k y_t, \forall 1 \leq k \leq K$ so that the {\rred{single-}UC} problem can be formulated as {an MILP} model, where $\mu_k^t = 2a {\ddot{x}_k} + b - q_t$ and $\nu_k = c-a {\ddot{x}_k}^2$ with ${\ddot{x}_k}$ {being} the $x-$value corresponding to the {$k$\textit{th}} supporting node on the curve of $f_t(x_t,y_t)$ at each time period $t$ and ${\ddot{x}_1}  = \Clower, \ {\ddot{x}_K} = \Cupper$. Therefore, Model \eqref{eqn:org_model_no_su_sd} can be reformulated as
\begin{eqnarray}
&\min  & \sum_{t=1}^{T} \bigg( \overline{U}u_t + \underline{U}(y_{t-1} - y_t + u_t) + \varphi_t \bigg) \nonumber \\
&\mbox{s.t.}& (x,y,u) \in \mathcal{D}, \nonumber \\
&& \varphi_t \geq \mu_k^t x_t + \nu_k y_t, \ \forall k \in [1,K]_{\Z}, \forall t \in [1,T]_{\Z}. \label{eqn:duc_piecewise_cons}
\end{eqnarray}
Note {here} that the cost function ${\varphi_t}$ is a linear function if there is only one piece, i.e., $K=1$. It is easy to observe that any two adjacent pieces at each time period $t$, e.g., $\varphi_t \geq \mu_k^t x_t + \nu_k y_t$ and $\varphi_t \geq \mu_{k+1}^t x_t + \nu_{k+1} y_t$, intersect at $A_k=({\ddot{x}_k} + {\ddot{x}_{k+1}})/2$. Therefore, we can obtain $K-1$ {turning} points with $x-$value $A_k$, {for all} $k \in [1, K-1]_{\Z}$ on the $K-$piece piecewise linear function for each time period. We denote $ {\Lambda_k} = \{n \in [1,T]_{\Z}: \Clower \leq A_k + nV \leq \Cupper\}$ and
\begin{eqnarray}
\mathcal{Q}_d = \mathcal{Q} \cup \{ (A_k+nV)_{n \in {\Lambda_k}}, \forall k \in [1,K-1]_{\Z} \}.
\end{eqnarray}

\begin{proposition} \label{prop:duc_opt_sol}
Problem \eqref{eqn:org_model_no_su_sd} has at least one optimal solution $(\bar{x},\bar{y},\bar{u})$ with $\bar{x}_t \in \mathcal{Q}_d$ for all $t \in [1,T]_{\Z}$.
\end{proposition}
\begin{proof}
{We prove this claim by a contradiction method.}
Suppose that there exists some $t \in [1,T]_{\Z}$ such that $\bar{x}_t \notin \mathcal{Q}_d$ for the optimal solution $(\bar{x}, \bar{y}, \bar{u})$ of Problem \eqref{eqn:org_model_no_su_sd}, i.e., $\bar{x}_t \in (\Clower, \Cupper) \setminus \mathcal{Q}_d$, with the optimal value $\bar{z} = \sum_{t=1}^{T}  \big( \overline{U} \bar{u}_t + \underline{U}(\bar{y}_{t-1} - \bar{y}_t + \bar{u}_t) +  \bar{\varphi}_t \big)$. In the following we construct a feasible solution to obtain a better objective value. If $t \geq 2$ and $|\bar{x}_t - \bar{x}_{t-1}| = V$, we let $s_1 \leq t-1$ be the smallest index such that $|\bar{x}_{s_1+1} - \bar{x}_{s_1}| = |\bar{x}_{s_1+2} - \bar{x}_{s_1+1}| = \cdots = |\bar{x}_t - \bar{x}_{t-1}| = V$; otherwise, we let $s_1 = t$. If $t \leq T-1$ and $|\bar{x}_{t+1} - \bar{x}_{t}| = V$, we let $s_2 \geq t+1$ be the largest index such that $|\bar{x}_{t+1} - \bar{x}_{t}| = |\bar{x}_{t+2} - \bar{x}_{t+1}| = \cdots = |\bar{x}_{s_2} - \bar{x}_{s_2-1}| = V$; otherwise, we let $s_2 = t$. Following the proof in Proposition \ref{prop:duc_ext_point}, we can always construct two feasible points of $\mathcal{D}$, $(\bar{x}^1, \bar{y}, \bar{u})$ and $(\bar{x}^2, \bar{y}, \bar{u})$, such that $\bar{x}^1_r = \bar{x}_r + \epsilon$ for $r \in [s_1,s_2]_{\Z}$, $\bar{x}^2_r = \bar{x}_r - \epsilon$ for $r \in [s_1,s_2]_{\Z}$, and $\bar{x}^1_r = \bar{x}^2_r = \bar{x}_r$ for $r \notin [s_1,s_2]_{\Z}$, where $\epsilon$ is an arbitrarily small positive number.

Now we show one of $(\bar{x}^1, \bar{y}, \bar{u})$ and $(\bar{x}^2, \bar{y}, \bar{u})$ produces a better objective value and is also feasible for constraints \eqref{eqn:duc_piecewise_cons}. It is easy to observe that for each time period $s$, if $\bar{x}_s > 0$, then at most two adjacent pieces of linear functions of constraints \eqref{eqn:duc_piecewise_cons} are tight; otherwise, it leads to $\bar{x}_s = 0$. Since $\bar{x}_t \notin \mathcal{Q}_d$, there is at most one piece of linear function of constraints \eqref{eqn:duc_piecewise_cons}, e.g., piece $k_r$, is tight for each time period $r \in [s_1, s_2]_{\Z}$, as two adjacent pieces (e.g., pieces $k$ and $k+1$) intersect at $A_k$, which belongs to $\mathcal{Q}_d$. Note that if there exist two adjacent pieces (e.g., pieces $k$ and $k+1$) intersecting at $A_k$ for some time period $r \in [s_1, s_2]_{\Z}$, then $\bar{x}_{r} = A_k$ and it follows that $\bar{x}_t = A_k+sV$ for some $s \in [-|t-r|, |t-r|]_{\Z}$ by definition, which contradicts the fact that $\bar{x}_t \notin \mathcal{Q}_d$ and $s \leq T$.
Therefore, we can increase or decrease $\bar{x}_r$ by $\epsilon$ for $r \in [s_1, s_2]_{\Z}$ to decrease the optimal value by at least $| \sum_{r=s_1}^{s_2} \mu_{k_r}^r | \epsilon$ and the resulting solution, $(\bar{x}^1, \bar{y}, \bar{u})$ or $(\bar{x}^2, \bar{y}, \bar{u})$, is feasible for both $\mathcal{D}$ and constraints \eqref{eqn:duc_piecewise_cons}, which is a contradiction.
\end{proof}

In other words, in order to find an optimal solution to the deterministic \rred{single-}UC problem \eqref{model:Nduc} with piecewise linear objective function (i.e., \eqref{eqn:org_model_no_su_sd}), we only need to consider the feasible solutions $(x,y,u)$ where $x_t \in \mathcal{Q}_d$ for all $t \in [1,T]_{\Z}$.

\subsection{An $\Oe(T)$ Time Dynamic Programming Algorithm} \label{subsec:duc-ot-dp}
As $ \mathcal{Q}_d $ is a finite set and the cardinality of $\mathcal{Q}_d$, {in the order of $\mathcal{O} ( (K+2) \lceil (\Cupper-\Clower)/V \rceil )$} and denoted as {$|\mathcal{Q}_d|$}, does not depend on the total time period $T$, we can explore a backward induction dynamic programming framework by defining the corresponding states and decisions for each time period $ t $. We first define the state space for the dynamic programming algorithm and then derive the corresponding Bellman equations based on the described state space.

\subsubsection*{The State Space}
The state space in general can be described as $(x,y,u,d)$ with $(x,y,u)$ as defined in Section \ref{Ndeteruc} and $d$ representing the duration of the generator at the current status (online or offline). For the convenience of later on analysis of extended formulations, we denote the state space as $ \mathcal{S} = \mathcal{S}_* \cup \mathcal{S}_0 \cup \mathcal{S}_1 $ with details provided as follows.
\begin{itemize}[nolistsep, leftmargin=*]
\item The set $ \mathcal{S}_* $, which is a singleton, includes a dummy source state as the initial status of the generator before time period $1$.
\item The set $\mathcal{S}_0$ includes all the states when the generator is offline, i.e., $ \mathcal{S}_0 = \{ (x, y, u, d) \in \mathcal{Q}_d \times \B \times \B \times [1, \ell]_{\Z}: x=0, y=0, u=0 \}$. Note here that we use $\ell$ as the upper bound for $d$ because the decision making remains the same whenever the duration $d \geq \ell $.
\item The set $\mathcal{S}_1$ includes all the states when the generator is online, i.e., $\mathcal{S}_1 = \{ (x, y, u,d) \in \mathcal{Q}_d \times \B \times \B \times [1, L]_{\Z}: y=1, u=1 \mbox{ when } d=1; \ y=1, u=0 \mbox{ when } d>1 \} $. Note here that we let $ d = 1 $ if the generator just starts up, i.e., $y=1$ and $u=1$. It is obvious that $y=1$ and $u=0 $ when $d > 1$. Similarly, it is sufficient to have $ d \leq L $.
\end{itemize}
Based on the {above} construction, we can observe that $ \mathcal{S}$ is a finite set and the cardinality of $\mathcal{S}$, i.e., $|\mathcal{S}|${,} {in the order of $\mathcal{O}(\ell + L (K+2) \lceil (\Cupper-\Clower)/V \rceil )$}, does not depend on the total time period $T$.

\subsubsection*{The State Transition Graph}
Now, we construct a directed graph, \rred{as shown in the following Figure \ref{fig:direct-graph-2},} to show all possible state transitions before deriving the Bellman equations. Essentially each state transition in the graph represents a decision from one state to another (i.e., correspondingly from time $t$ to time $t+1$). In particular, based on the state space $\mathcal{S}$, we add directed arcs as follows to link different states, representing all possible state transitions.
\begin{itemize}[nolistsep, leftmargin=*]
\item First, we add directed arcs from the source state in $\mathcal{S}_*$ to state $(0, 0, 0, \ell)$ from $\mathcal{S}_0$ and each state $(x,1,0,L) \in \mathcal{S}_1$ with $x \in \mathcal{Q}_d$. Note {here} that these to-go states of the source state form the set of decision candidates at time $1$.
\item Second, we add possible directed arcs between any two states in $\mathcal{S}_0 \cup \mathcal{S}_1$ as long as the state transition {satisfies} the minimum-up/-down time, ramp rate, and capacity restrictions. In particular, based on the states indicating the generator is offline, we add arcs from $ (0,0,0,s) $ to $ (0,0,0,s+1) $ for all $s \in [1,\ell - 1]_{\Z}$, from $ (0,0,0,\ell) $ to both itself and each state $(x,1,1,1) \in \mathcal{S}_1$ with $x \in \mathcal{Q}_d {\cap [0, \Vupper] }$. Note here that self-loop arc is allowed since it is enough to take $ d $ as $ \ell $ if the offline status $ d $ lasts longer than the minimum-down time limit $\ell$. In addition, for the state indicating the generator is online, we have the following three possible cases.
	\begin{itemize}[nolistsep]
	\item For each state $(x,y,u,d) \in  \mathcal{S}_1$ with $d<L$, we add an arc from it to each state $ (\bar{x},y,\bar{u},d+1)  \in \mathcal{S}_1$ such that $|x - \bar{x}| \leq V$.
	\item For each state $ (x,y,u,L) \in  \mathcal{S}_1$, we add an arc from it to each state $ (\bar{x},y,u,L )  \in \mathcal{S}_1$ such that $  |x - \bar{x}| \leq V $. Self-loop {arc} is also allowed here.
	\item For each state $ (x,y,u,L) \in  \mathcal{S}_1$ with $ x  \leq \Vupper $, i.e., the shut-down ramp constraint is satisfied, we add an arc from it to $ (0,0,0,1) $.
	\end{itemize}
\end{itemize}
Furthermore, we label all the states in $\mathcal{S}$ with positive integers starting from the source state with index $0$. Meanwhile, we denote all the immediate successors of state $ i $ as $ S(i) $ and all the {immediate} predecessors of state $ i $ as $ P(i) $. Also, we denote the values of $(x,y,u)$ corresponding to state $ i $ as $(x(i), y(i), u(i))$.


\begin{figure}[htb]
	\centering
	\begin{tikzpicture}[ scale=0.56, every node/.style={scale=0.7}] 
	
	\tikzstyle{state} = [ draw, fill=white, circle, text width = 4em, 
	text badly centered, node distance=8em, inner sep=1pt, font=\sffamily]
	\tikzstyle{stateEdgePortion} = [black];
	\tikzstyle{stateEdge} = [stateEdgePortion,-{>[scale=2.5,
          length=2,
          width=2]}];
	\tikzstyle{edgeLabel} = [pos=0.5, text centered, font={\sffamily}];

\node[state, name=10] at (0,0) {\small $(x,1,1,1)$}; 
\node[state, name=20] at (7,0)  {\small $(x,y,u,L)$};
\node[state, name=30] at (14,0)  {\small $(\bar{x},y,u,L)$};
\node[state, name=40] at (21,0) {\small $(x,1,0,L)$};

\node[state, name=1-7] at (0,-7) {\small $(x,1,1,1)$};
\node[state, name=2-7] at (7,-7) {\small $(x,y,u,d)$};
\node[state, name=3-7] at (14,-7) {\small $(\bar{x},y,\bar{u},d+1)$};
\node[state, name=4-7] at (21,-7) {\small $(x,1,0,L)$};

\node[state, name=4-12] at (21,-12) {\small $(0,0,0,1)$ };
\node[state, name=3-12] at (14,-12) {\small $(0,0,0,s)$};
\node[state, name=2-12] at (7,-12) {\small $(0,0,0,s+1)$};
\node[state, name=1-12] at (0,-12) {\small $(0,0,0,\ell)$};

\node[state, name=0-12] at (24.5,-14) {$0$};
\draw (24.5,-15.5) node{ \mbox{Source state}};

	\draw (0-12) edge[stateEdge,  bend left=15] (1-12);
	\draw (0-12) edge[stateEdge,  bend left=-25] (4-7);
	\draw (0-12) edge[stateEdge,  bend left=-30, dashdotted ] (22,-3.5);
	\draw (0-12) edge[stateEdge,  bend left=-35] (40);
	
	\draw (3-12) edge[stateEdge ]  (2-12);
	\draw (4-12) edge[stateEdge, dashdotted ]  (18.5,-12);
	\draw (17.5,-12) node{ \Large {$\cdots$}  };
	\draw (16.5,-12) edge[stateEdge, dashdotted ]  (3-12);
	\draw (2-12) edge[stateEdge, dashdotted ]  (4.5,-12);
	\draw (3.5,-12) node{ \Large {$\cdots$}  };
	\draw (2.5,-12) edge[stateEdge, dashdotted ]  (1-12);
	\draw (1-12) edge[stateEdge,  loop below ] (1-12);

	\draw (1-12) edge[stateEdge,  bend left=35] (1-7);
	\draw (.8,-8.7) node{$x \in \mathcal{Q}_d \cap [0, \Vupper] $  };
	\draw (1-12) edge[stateEdge,  bend left=45, dashdotted] (-1,-5);

	\draw (2-7) edge[stateEdge ] (3-7);
	\draw (10.5,-7.8) node{\large $|x - \bar{x}| \leq V$ };
	\draw (1-7) edge[stateEdge, dashdotted ]  (2.5,-7);
	\draw (3.5,-7) node{ \Large {$\cdots$}  };
	\draw (4.5,-7) edge[stateEdge, dashdotted ]  (2-7);
	\draw (3-7) edge[stateEdge, dashdotted ]  (16.5,-7);
	\draw (17.5,-7) node{ \Large {$\cdots$}  };
	\draw (18.5,-7) edge[stateEdge, dashdotted ]  (4-7);
	\draw (4-7) edge[stateEdge,  loop right ] (4-7);
	\draw (1-7) edge[stateEdge, dashdotted ]  (2.5,-6);
	\draw (4.5,-6) edge[stateEdge, dashdotted ]  (2-7);
	\draw (3-7) edge[stateEdge, dashdotted ]  (16.5,-6);
	\draw (18.5,-6) edge[stateEdge, dashdotted ]  (4-7);

	\draw (20) edge[stateEdge ] (30);
	\draw (10.5,-.8) node{\large $|x - \bar{x}| \leq V$ };
	\draw (10) edge[stateEdge, dashdotted ]  (2.5,0);
	\draw (3.5,0) node{ \Large {$\cdots$}  };
	\draw (4.5,0) edge[stateEdge, dashdotted ] (20);
	\draw (30) edge[stateEdge, dashdotted ]  (16.5,0);
	\draw (17.5,0) node{ \Large {$\cdots$}  };
	\draw (18.5,0) edge[stateEdge, dashdotted ]  (40);	
	\draw (40) edge[stateEdge,  loop right ] (40);
	\draw (10) edge[stateEdge, dashdotted ]  (2.5,-1);
	\draw (4.5,-1) edge[stateEdge, dashdotted ] (20);
	\draw (30) edge[stateEdge, dashdotted ]  (16.5,-1);
	\draw (18.5,-1) edge[stateEdge, dashdotted ]  (40);	

	\draw (40) edge[stateEdge, bend left = 10, dashdotted ]  (21.5,-2);
	\draw (20.5,-2) edge[stateEdge, bend left = 10, dashdotted ]  (40);
	\draw (21.5,-5) edge[stateEdge, bend left = 10, dashdotted ]  (4-7);
	\draw (4-7) edge[stateEdge, bend left = 10, dashdotted ]  (20.5,-5);	

	\draw (4-7) edge[stateEdge,  bend left=35] (4-12);
	\draw (21.5,-9.5) node{\large $x \leq \Vupper$ };

	\draw (0,-3.5) node{ \Large {$\cdots$}  };
	\draw (7,-3.5) node{ \Large {$\cdots$}  };
	\draw (14,-3.5) node{ \Large {$\cdots$}  };
	\draw (21,-3.5) node{ \Large {$\cdots$}  };
	\draw (3.5,-3.5) node{ \Large {$\cdots$}  };
	\draw (17.5,-3.5) node{ \Large {$\cdots$}  };

	\end{tikzpicture}
	\caption{The State Transition Graph}\label{fig:direct-graph-2}
\end{figure}


\subsubsection*{The Bellman Equations}
Now we are ready to establish the dynamic programming framework. We let $ F_t (i)$ represent the optimal value function for time period $ t $ considering that state $i$ happens at time $t-1$. Based on Proposition \ref{prop:duc_opt_sol}, an optimal decision for current time period lies in $ S(i) $. {Accordingly,} the Bellman equations can be formulated as follows:
\begin{subeqnarray} \label{duc_dp}
	&& F_t (i) = \min_{j\in S(i)} \bigg\{ \overline{U}u(j) + \underline{U}(y(i) - y(j) +u(j)) + f_t(x(j), y(j))  + F_{t+1}(j) \bigg\}, \nonumber \\
	&& \hspace{4in} \forall i \in  \mathcal{S}, \forall t \in [1, T-1]_{\Z}, \slabel{eqn:dp1}\\
	&& F_T (i) = \min_{j\in S(i)} \bigg\{ \overline{U}u(j) + \underline{U}(y(i) - y(j) +u(j)) + f_T(x(j), y(j))  \bigg\}, \ \forall i \in  \mathcal{S},	\slabel{eqn:dp3}
\end{subeqnarray}
where $ f_t(x(j), y(j)) $ describes the generation cost minus the revenue, $ \overline{U}u(j) $ represents the start-up cost, and $ \underline{U}(y(i) - y(j) +u(j)) $ represents the shut-down cost. For notation brevity, we let $ E_{tij} = \overline{U}u(j) + \underline{U}(y(i) - y(j) +u(j)) + f_t(x(j), y(j))  $ for $ t \in [1, T]_{\Z} $. {Accordingly,} the objective of our backward induction dynamic programming is to find out the value of {$ F_1(0) $}.

\begin{proposition} \label{prop:duc_order_T}
The deterministic \rred{single-}UC problem \eqref{eqn:org_model_no_su_sd} can be solved in $ \Oe(T) $ time, \rred{or more specifically, in $ \Oe ({|\mathcal{Q}_d|}({|\mathcal{Q}_d|}L+\ell) T) $} time.
\end{proposition}
\begin{proof}
In order to obtain the optimal objective value and optimal solution, we need to calculate $ F_t (i) $ for all possible $t$ and $i$ and record the optimal candidates for them. To calculate the value of each optimal value function in Bellman equations \eqref{duc_dp}, we search among the candidate {solutions} $ j\in S(i) $ for each $ i $ and this step takes $ \Oe({|\mathcal{Q}_d|}) $ time. Since there are in total $ {|\mathcal{Q}_d|}L+\ell $ number of states in the state graph, the computational time at each time period is $\Oe ({|\mathcal{Q}_d|}({|\mathcal{Q}_d|} L+\ell) )$. Thus, the total time to calculate the value of objective $  F_1(0) $ is $ \Oe ({|\mathcal{Q}_d|}({|\mathcal{Q}_d|}L+\ell) T) $. The optimal solution for the problem can be obtained accordingly. 
\end{proof}


\rred{Note here that from Proposition 5, we can observe that for a given $T$, when a generator has a very small ramping rate but a large generation capacity difference between the lower and upper bounds (e.g., $V << \overline{C} - \underline{C}$ and thus large $|\mathcal{Q}_d|$), the corresponding single-UC problem will be relatively more difficult to solve.}

When start-up profile is considered, we have the following observations.
\begin{remark} \label{rmk:duc-su}
	If the start-up profile (i.e., time-dependent start-up cost) is taken into account in the deterministic \rred{single-}UC model \eqref{eqn:org_model_no_su_sd}, then we need to extend the upper bound of the offline duration variable $ d $ from $\ell$ to {the cold-start point of the generator}, which as a result {keep{s}} the computational complexity as $\Oe(T)$ in general.
\end{remark}

\subsection{An Extended Formulation for Deterministic \rred{Single-}UC with Piecewise Linear Cost Function} \label{exuc}
Based on Bellman equations \eqref{duc_dp} {derived} for the deterministic \rred{single-}UC model, we reformulate the problem as a linear program {extended formulation}. The incentive for the reformulation is to develop a linear program that can provide integral solutions to the deterministic \rred{single-}UC problem in a higher dimensional space. The primal form of our linear program can be formulated as follows:
\begin{subeqnarray} \label{model:duc_lp_bell}
	&\max  & F_1(0) \\
	&\mbox{s.t.}  & F_t(i) \leq E_{tij} + F_{t+1}(j), \ \forall i \in \mathcal{S}, j \in S(i), \forall t \in [1, T-1]_{\Z},  \\
	&& F_T(i) \leq E_{Tij}, \ \forall i \in \mathcal{S}, j \in S(i),
\end{subeqnarray}
where the parameters $E_{tij} =  \overline{U}u(j) + \underline{U}(y(i) - y(j) +u(j)) + f_t(x(j), y(j)) $ for all possible $(t, i, j)$ are defined under equations \eqref{duc_dp}. 

The formulation \eqref{model:duc_lp_bell} cannot provide {an optimal solution} to the deterministic \rred{single-}UC problem directly. Although we can solve {formulation \eqref{model:duc_lp_bell}} and search for tight constraints to find solutions, this takes us back to the dynamic programming framework. Thus, we resort to its dual formulation and then provide an extended linear formulation of the original problem, which we show can provide {an optimal} solution to the deterministic \rred{single-}UC problem directly.
The dual formulation for \eqref{model:duc_lp_bell} {can be described} as follows:
\begin{subeqnarray}  \label{model:duc_d}
	&\min  & \sum_{t \in [1, T]_{\Z}, i \in \mathcal{S}, j \in S(i)} E_{tij}w_{tij} \slabel{eqn:dual0}\\
	&\mbox{s.t.}  & \sum_{j \in S(1)} w_{11j} = 1, \slabel{eqn:dual1}\\
	&& \sum_{j \in S(i)} w_{1ij} = 0, \ \forall i \in \mathcal{S} \setminus \mathcal{S}_*,\slabel{eqn:dual2} \\	
	&& \sum_{j \in S(i)} w_{tij} - \sum_{k \in P(i)} w_{t-1,ki} = 0, \ \forall  i \in \mathcal{S}, \forall t \in [2, T]_{\Z}, \slabel{eqn:dual3} \\	
	&& w_{tij} \geq 0, \ \forall i \in \mathcal{S}, j \in S(i), \forall  t \in [1, T]_{\Z}, \slabel{eqn:dual4}
\end{subeqnarray}
where $w_{tij} $ for each possible $(t, i, j)$ is {the} dual variable corresponding to each constraint in \eqref{model:duc_lp_bell}.

In the following lemma, we demonstrate that the above dual program \eqref{model:duc_d} can automatically generate integral solutions with respect to $ w $. Furthermore, we explore physical {meanings} of these dual variables.

\begin{lemma} \label{lemma:binary solution}
	The extreme points of the polytope \eqref{eqn:dual1} -- \eqref{eqn:dual4} are binary.
\end{lemma}
\begin{proof}
To prove the lemma, it is equivalent to prove that $\forall \ E_{tij} \in {\R} $, there exists at least one optimal solution to the dual program \eqref{model:duc_d} that is binary.

First of all, by solving the deterministic \rred{single-}UC problem with dynamic program approach with respect to a given $ E $, we can obtain an optimal decision of the online/offline status and generation amount of the generator at each time period, i.e., $(\bar{x}_t, \bar{y}_t, \bar{u}_t)$ for all $t \in [1,T]_{\Z}$. We then construct $ \hat{w} $, which is binary, to represent this optimal decision.
Based on this optimal decision $ (\bar{x}, \bar{y}, \bar{u}) \in \mathcal{Q}_d^{T} \times \B^{2T}$, we can accordingly draw a directed path in our constructed {state transition} graph in Subsection \ref{subsec:duc-ot-dp}
 starting from state $ 1 $ and covering $T$ time periods. If in step $ t $ (i.e., at time period $t$) of the path the state of the generator goes from state $ i $ to state $ j $, we let $  \hat{w}_{tij} = 1 $. Otherwise, $  \hat{w}_{tij} = 0 $. In the following, we prove that this constructed solution $ \hat{w} $ is an optimal solution to the dual program \eqref{model:duc_d}, which is sufficient to prove our claim.

{To verify the feasibility, we plug $ \hat{w} $ into constraints \eqref{eqn:dual1} -- \eqref{eqn:dual4} to show {that} each constraint holds. For constraint \eqref{eqn:dual1}, since at time period $1$ the state of the generator goes from state $ 1 $ to only one of the states in set $ S(1) $, we have $  \hat{w}_{11 \hat{j}} = 1 $ for some $\hat{j} \in S(1)$ and $  \hat{w}_{11j} = 0 $ for all $j \in  S(1)$ such that $j \neq \hat{j}$. Constraint \eqref{eqn:dual2} is satisfied since at time period $1$ the generator has to start from dummy state $1 \in \mathcal{S}_*$ and change its state to some other one in $\mathcal{S} \setminus \mathcal{S}_*$, i.e., $  \hat{w}_{1ij} = 0, \ \forall i \in \mathcal{S} \setminus \mathcal{S}_* $. 
For each time $t \in [2,T]_{\Z}$ and state $i \in \mathcal{S}$ in constraints \eqref{eqn:dual3}, if the state of the generators changes from state $i$ to some state $j \in S(i)$, i.e., $\sum_{j \in S(i)} w_{tij} = 1$, then we have the state of the generator must change from some state $k \in P(i)$ to state $i$ at time $t-1$, i.e., $\sum_{k \in P(i)} w_{t-1,ki} = 1$; otherwise, we have  $\sum_{j \in S(i)} w_{tij} =\sum_{k \in P(i)} w_{t-1,ki} = 0$. For both cases, we have constraints \eqref{eqn:dual3} hold.
Constraint \eqref{eqn:dual4} is satisfied due to the construction of $\hat{w}$. Thus, $ \hat{w}$ is feasible for the dual program.}

To verify the optimality, we denote the optimal objective value obtained by dynamic program as $ F^* $ and the objective value of dual program \eqref{model:duc_d} with respect to $ \hat{w} $ as $ H(\hat{w}) $. We need to show that $ H(\hat{w}) = F^* $. Recalling the definition of $ \hat{w} $, we use a path, denoted as $ (i_0,i_1, i_2, \ldots, i_T) $, to determine $ \hat{w} $. Thus, we have
	\begin{eqnarray} 
		& H(\hat{w}) & = \sum_{t \in [1, T]_{\Z}, i \in \mathcal{S}, j \in S(i)} E_{tij}\hat{w}_{tij} = \sum_{t =1}^T E_{ti_{t-1}i_t}\hat{w}_{ti_{t-1}i_t} = \sum_{t =1}^T E_{ti_{t-1}i_t}, \mbox{ and } 
	\end{eqnarray}
	\begin{eqnarray} 
		& F^* = F_1 (1) & = \min_{j\in S(1)} \Big\{ E_{11j} + F_2(j) \Big\} = E_{1i_0 i_1} +F_2 (i_1) \nonumber \\
		&& = E_{1i_0 i_1} + \min_{j\in S(i_1)} \Big\{ E_{2i_1 j} + F_3(j) \Big\} = E_{1i_0 i_1} + E_{2i_1 i_2} + F_3(i_2) \nonumber \\
		&& = \cdots  = \sum_{t =1}^T E_{ti_{t-1}i_t} =  H(\hat{w}). 
	\end{eqnarray}
	Therefore, we claim that $ \hat{w} $ is the optimal solution to the dual program.
	
	From the above analysis, we notice that  $ \hat{w} $ is binary and optimal for the dual program {for any possible} cost coefficient $ E $. Thus, we have proved our claim.
\end{proof}

From Lemma \ref{lemma:binary solution}, we can further observe that this formulation itself has specific physical meanings. In particular, we notice that the optimal solution $ w^*_{tij} $ in the dual program represents {whether} at time period $ t $ the optimal decision corresponds to a state change from state $i$ (at time period $t-1$) to state $j$ (at time period $t$) or not. If yes, then $ w^*_{tij} = 1 $; otherwise, $ w^*_{tij} = 0 $. That is, the lemma provides an insight to formulate the deterministic \rred{single-}UC problem in a different way by linking the decision variables in the original space and the decision variables in the dual formulation \eqref{model:duc_d}. In the following, we present the detailed extended formulation for the deterministic \rred{single-}UC problem and justify its correctness.

\begin{proposition} \label{thm: opt solution}
	If $ w^* $ is an optimal solution to {the} dual program \eqref{model:duc_d}, then
	\begin{eqnarray} \label{eqn:opt solution}
		&& x_t^* = \sum_{i \in \mathcal{S}, j \in S(i)} x(j) w_{tij}^*, y_t^* = \sum_{i \in \mathcal{S}, j \in S(i)} y(j) w_{tij}^*, u_t^* = \sum_{i \in \mathcal{S}, j \in S(i)} u(j) w_{tij}^*, \ \forall t \in  [1, T]_{\Z}, \slabel{eqn:opt solution1} 
	\end{eqnarray}
	is an optimal solution to the deterministic \rred{single-}UC problem \eqref{eqn:org_model_no_su_sd}.
\end{proposition}
\begin{proof}
{From the proof of Lemma \ref{lemma:binary solution}, we can observe that for a given $ t $, $ w_{tij}^* =1$ for exactly one possible pair $(i,j)$ with $i, j \in \mathcal{S}$ and $ w_{tij}^* = 0 $ for all other possible pairs. All the pairs $i, j \in \mathcal{S}$ such that $ w_{tij}^* =1$ for all $t \in [1,T]_{\Z}$ construct an optimal path with $T$ steps originating from state $1$. Meanwhile, this path indicates the optimal online/offline status and generation amount of the generator at each time period $t$, i.e., $(x_t^*, y_t^*, u_t^*)$. 
	If we denote the entire path as $ (i_0,i_1, i_2, \ldots, i_T) $, {then we have} that $ x_t^* = x(i_t),y_t^* = y(i_t)$, and $u_t^* = u(i_t)$ {for each $t \in [1,T]_{\Z}$.} Recalling the definition of the {state transition} graph we can conclude that any ``T-step'' path {(i.e., a path {covering} the whole time horizon)} starting from state $ 1 $ results in a feasible solution to the deterministic \rred{single-}UC problem. Thus, the solution $ (x^*,y^*,u^*) $ in \eqref{eqn:opt solution1} is feasible.} 
	
	If we {replace} $ (x^*,y^*,u^*) $ with {$ w_{tij}^* $} {as described in the expression} \eqref{eqn:opt solution} into the objective function of the deterministic \rred{single-}UC problem \eqref{eqn:org_model_no_su_sd}, we have
	\begin{eqnarray} 
		\sum_{t=1}^{T} \bigg( \overline{U}u_t^* + \underline{U}(y_{t-1}^* - y_t^* + u_t^*) + f(x_t^*,y_t^*) \bigg) =  \sum_{t \in [1, T]_{\Z}, i \in \mathcal{S}, j \in S(i)} E_{tij}w_{tij}^*. \nonumber
	\end{eqnarray}
	Hence,  $ (x^*,y^*,u^*) $ is the optimal solution to the deterministic \rred{single-}UC problem.
\end{proof}

Now we are ready to {present} the extended formulation for the deterministic {\rred{single-}UC} problem with piecewise linear cost function. We replace constraints \eqref{eqn:Np-minup} - \eqref{eqn:Np-ramp-down} plus  \eqref{eqn:Np-nonnegativity} with constraints \eqref{eqn:dual1} -- \eqref{eqn:dual4} and add equations \eqref{eqn:opt solution}  to represent the relation between original decisions and the dual decision variables.
\begin{theorem}\label{thm:extend}
	The extended formulation of the deterministic {\rred{single-}UC} problem  \eqref{eqn:org_model_no_su_sd} can be written as follows:
	\begin{eqnarray} 
		&\min  & \sum_{t=1}^{T} \bigg( \overline{U}u_t + \underline{U}(y_{t-1} - y_t + u_t) + \rblue{\varphi_t} \bigg) \nonumber \\
		&\mbox{s.t.}  & x_t = \sum_{i \in \mathcal{S}, j \in S(i)} x(j) w_{tij}, \ y_t = \sum_{i \in \mathcal{S}, j \in S(i)} y(j) w_{tij}, \ u_t = \sum_{i \in \mathcal{S}, j \in S(i)} u(j) w_{tij}, \ \forall t \in  [1, T]_{\Z}, \nonumber \\
		&& \rblue{\eqref{eqn:duc_piecewise_cons},}  \ 	\eqref{eqn:dual1}- \eqref{eqn:dual4}, 
	\end{eqnarray}
	and if $ (x^*,y^*,u^*,w^*) $ is an optimal solution to the extended formulation, then  $ (x^*,y^*,u^*) $ is an optimal solution to the deterministic \rred{single-}UC problem.
\end{theorem}
\begin{proof}
	The conclusion holds immediately by replacing $ (x, y, u) $ in the objective function with {$ w_{tij} $ as described in} the expression \eqref{eqn:opt solution}. {Then the above formulation converts to model \eqref{model:duc_d}. Thus the conclusion holds} following the proof described in Proposition \ref{thm: opt solution} and the conclusion described in Lemma \ref{lemma:binary solution}.
\end{proof}

\section{Stochastic Single-UC with Piecewise Linear Cost Function} \label{polysuc}
In this section, we extend our study to the multistage stochastic \rred{single-}UC setting to incorporate uncertainty. {Similar to Section \ref{deteruc}, we will explore the optimality condition of this problem, develop an efficient dynamic programming algorithm to solve the problem, and derive an extended formulation in a higher dimensional space that provides the integral solutions. To that end, we first provide the mathematical formulation of this problem as follows.}

With the consideration of renewable {energy} generation and/or price uncertainties, as well as dependency among different time periods, {the} scenario-tree based stochastic \rred{single-}UC is introduced in \cite{jiang2016cutting}.
 Under this setting, the uncertain problem parameters are assumed to follow a discrete-time stochastic process with finite probability space and a scenario tree $\mathcal{T} = (\mathcal{V}, \mathcal{E})$ is utilized to describe the resulting information structure, as shown in Figure~\ref{fig:tree}. Each node $m \in \Ve$ at time $t$ of the tree provides the state of the system that can be distinguished by information available up to {time} $t$ ({corresponding to a scenario realization from times $1$ to $t$}).
Accordingly, corresponding to each node $m \in \Ve$, we let $t(m)$ be its time period, $\Pe(m)$ be the set of nodes along the path from the root node (denoted as node $ 1 $) to node $m$, and {$p_m$ be the probability associated with the state represented by node $m$.} We also denote {$ \Ve_* $ as the set of root node, i.e., $ \Ve_* = \{1\}$.} In addition, each node $m$ in the scenario tree has a unique parent $m^-$ and could have multiple children, denoted as set $\C_*(m)$. {Meanwhile, we define} $ \C(m) = \C_*(m) \cup \{m\} $. {Moreover}, we let $m^-_0 = m$, $m^-_1 = m^-$, and $m^-_k$ be the unique parent node of $m^-_{k-1}$, for $k \geq 2$.
We let $\Ve(m)$ represent the set of all descendants of node $m$, including itself. Finally, we let $\He_r(m) = \{ k \in \Ve(m): 0 \leq t(k) - t(m) \leq r-1 \} $ be the set of nodes used to describe minimum-up and minimum-down time constraints {(e.g., in Figure~\ref{fig:tree}, $r = {t(j)} - {t(m)}$)}. The decisions corresponding to each node $m$ are assumed to be made after observing the realizations of the problem parameters along the path from the root node to this node $m$, but are nonanticipative with respect to future realizations.
\begin{figure}[!htbh]
\begin{center}
\begin{tikzpicture}[scale=0.5]
    \draw[line width=0.5pt] (-4, 0) -- (-2, 0);
    \draw[line width=0.5pt, dashed] (-2, 0) -- (0, 0);
    \draw[line width=0.5pt] (0, 0) -- (2, 0);
    \draw[line width=0.5pt] (2, 0) -- (4, 3);
    \draw[line width=0.5pt] (2, 0) -- (4, -3);
    \draw[line width=0.5pt] (4, 3) -- (6, 3);
    \draw[line width=0.5pt] (4, -3) -- (6, -3);
    \draw[line width=0.5pt, dashed] (6, 3) -- (8, 3);
    \draw[line width=0.5pt, dashed] (6, -3) -- (8, -3);
    \draw[line width=0.5pt] (8, 3) -- (10, 3);
    \draw[line width=0.5pt] (8, -3) -- (10, -3);

    \draw[line width=0.5pt] (10, 3) -- (12, 5);
    \draw[line width=0.5pt] (10, 3) -- (12, 1);
    \draw[line width=0.5pt] (10, -3) -- (12, -5);
    \draw[line width=0.5pt] (10, -3) -- (12, -1);

    \draw[line width=0.5pt] (12, 5)-- (14, 5);
    \draw[line width=0.5pt] (12, 1)-- (14, 1);
    \draw[line width=0.5pt] (12, -5)-- (14, -5);
    \draw[line width=0.5pt] (12, -1)-- (14, -1);

    \draw[line width=0.5pt, dashed] (14, 5)-- (16, 5);
    \draw[line width=0.5pt, dashed] (14, 1)-- (16, 1);
    \draw[line width=0.5pt, dashed] (14, -5)-- (16, -5);
    \draw[line width=0.5pt, dashed] (14, -1)-- (16, -1);

    \draw[line width=0.5pt] (16, 5)-- (18, 5);
    \draw[line width=0.5pt] (16, 1)-- (18, 1);
    \draw[line width=0.5pt] (16, -5)-- (18, -5);
    \draw[line width=0.5pt] (16, -1)-- (18, -1);

    \draw[line width=0.5pt] (18, 5)-- (20, 5);
    \draw[line width=0.5pt] (18, 1)-- (20, 1);
    \draw[line width=0.5pt] (18, -5)-- (20, -5);
    \draw[line width=0.5pt] (18, -1)-- (20, -1);

    \draw(-4, 0) circle (0.5cm);
    \fill[white] (-4, 0) circle (0.46cm);
    \draw (-4,-0.5) node[anchor=south] {$1$};

    \draw(-2, 0) circle (0.5cm);
    \fill[white] (-2, 0) circle (0.46cm);

    \draw(0, 0) circle (0.5cm);
    \fill[white] (0, 0) circle (0.46cm);

    \draw(2, 0) circle (0.5cm);
    \fill[white] (2, 0) circle (0.46cm);
    \draw (2,-0.5) node[anchor=south] {$m$};

    \draw(4, 3) circle (0.5cm);
    \fill[white] (4, 3) circle (0.46cm);

    \draw(6, 3) circle (0.5cm);
    \fill[white] (6, 3) circle (0.46cm);

    \draw(8, 3) circle (0.5cm);
    \fill[white] (8, 3) circle (0.46cm);

    \draw(10, 3) circle (0.5cm);
    \fill[white] (10, 3) circle (0.46cm);
    \draw (10,3.5) node[anchor=north] {$j$};

    \draw(4, -3) circle (0.5cm);
    \fill[white] (4, -3) circle (0.46cm);

    \draw(6, -3) circle (0.5cm);
    \fill[white] (6, -3) circle (0.46cm);

    \draw(8, -3) circle (0.5cm);
    \fill[white] (8, -3) circle (0.46cm);

    \draw(10, -3) circle (0.5cm);
    \fill[white] (10, -3) circle (0.46cm);
    \draw (10,-3.5) node[anchor=south] {$k$};

    \draw(12, 5) circle (0.5cm);
    \fill[white] (12, 5) circle (0.46cm);
    \draw(14, 5) circle (0.5cm);
    \fill[white] (14, 5) circle (0.46cm);
    \draw(16, 5) circle (0.5cm);
    \fill[white] (16, 5) circle (0.46cm);
    \draw(18, 5) circle (0.5cm);
    \fill[white] (18, 5) circle (0.46cm);
    \draw(20, 5) circle (0.5cm);
    \fill[white] (20, 5) circle (0.46cm);

    \draw(12, -1) circle (0.5cm);
    \fill[white] (12, -1) circle (0.46cm);
    \draw(14, -1) circle (0.5cm);
    \fill[white] (14, -1) circle (0.46cm);
    \draw(16, -1) circle (0.5cm);
    \fill[white] (16, -1) circle (0.46cm);
    \draw(18, -1) circle (0.5cm);
    \fill[white] (18, -1) circle (0.46cm);
    \draw(20, -1) circle (0.5cm);
    \fill[white] (20, -1) circle (0.46cm);

    \draw(12, 1) circle (0.5cm);
    \fill[white] (12, 1) circle (0.46cm);
    \draw(14, 1) circle (0.5cm);
    \fill[white] (14, 1) circle (0.46cm);
    \draw(16, 1) circle (0.5cm);
    \fill[white] (16, 1) circle (0.46cm);
    \draw(18, 1) circle (0.5cm);
    \fill[white] (18, 1) circle (0.46cm);
    \draw(20, 1) circle (0.5cm);
    \fill[white] (20, 1) circle (0.46cm);

    \draw(12, -5) circle (0.5cm);
    \fill[white] (12, -5) circle (0.46cm);
    \draw(14, -5) circle (0.5cm);
    \fill[white] (14, -5) circle (0.46cm);
    \draw(16, -5) circle (0.5cm);
    \fill[white] (16, -5) circle (0.46cm);
    \draw(18, -5) circle (0.5cm);
    \fill[white] (18, -5) circle (0.46cm);
    \draw(20, -5) circle (0.5cm);
    \fill[white] (20, -5) circle (0.46cm);

    \draw[line width=0.5pt, dashed] (9.3, -2) -- (11, -0.3) -- (21, -0.3) -- (21, -5.7) -- (11, -5.7) -- (9.3, -4) -- (9.3, -2);
    \draw (14.2, -3.8) node[anchor=south] {$\Ve(k)$};

    \draw[line width=0.5pt, dashed] (9.3, 2) -- (11, 0.3) -- (13, 0.3) -- (13, 5.7) -- (11, 5.7) -- (9.3, 4) -- (9.3, 2);
    \draw (12, 3.8) node[anchor=north] {$\C(j)$};

    \draw[line width=0.5pt, dashed] (1.2, -4) -- (8.7, -4) -- (8.7, 4) -- (1.2, 4) -- (1.2, -4);
    \draw (5, 4.3) node[anchor=south] {$\He_{r}(m)$};

    \draw (2, -7.5) node[anchor=north] {$\mbox{\small Time } {t(m)}$};
    \draw (10, -7.5) node[anchor=north] {$\mbox{\small Time } {t(j)}$};
    \draw (-4, -7.5) node[anchor=north] {$\mbox{\small Time } 1$};
    \draw (20, -7.5) node[anchor=north] {$\mbox{\small Time } T$};

\end{tikzpicture}
\caption{Multistage stochastic scenario tree} \label{fig:tree}
\end{center}
\end{figure}

For the multistage stochastic \rred{single-}UC problem, following the notation described above,
besides \rred{$\overline{U}$, $\underline U$, and $f(\cdot)$} defined as the deterministic case, for each node $m \in \mathcal{V}$, we let $(x_m, y_m, u_m)$ denote the generation amount, online/offline status, and start-up decision of the generator. Here we assume the generator has been offline for $ s_0 $ time periods ($ s_0 \geq \ell $) before node $1$. Accordingly, the formulation for this problem {can be described} as follows:
\begin{subeqnarray} \label{eqn:MSS}
&\min  & \sum_{m \in \Ve } p_m \bigg( \overline{U}u_m + \underline{U}(y_{m^-} - y_m + u_m) + f_m(x_m,y_m) \bigg) \slabel{eqn:objective} \\
&\mbox{s.t.}  & y_m - y_{m^-} \leq y_k, \ \ \forall m \in \Ve, \forall k \in \He_{L}(m), \slabel{eqn:min-up} \\
&  & y_{m^-} - y_m \leq 1- y_k, \ \ \forall
m \in \Ve, \forall k \in \He_{\ell}(m), \slabel{eqn:min-down} \\
&& y_{m} - y_{m^-} \leq u_{m}, \ \ \forall m \in \Ve, \slabel{eqn:u-def-1} \\
&& u_m \leq \min\{ y_m, \ 1 - y_{m^-} \}, \ \ \forall m \in \Ve, \slabel{eqn:u-def-2} \\
&& \underline{C} y_m \leq x_m \leq \overline{C} y_m, \ \ \forall m \in \Ve, \slabel{eqn:generator-bound} \\
&& x_{m} - x_{m^-} \leq V y_{m^-} + \overline{V} (1 - y_{m^-}), \ \ \forall m \in \Ve, \slabel{eqn:ramp-up} \\
&& x_{m^-} - x_{m} \leq V y_m + \overline{V} (1 - y_m), \ \ \forall m \in \Ve, \slabel{eqn:ramp-down} \\
&& y_m,u_m \in \{0, 1\}, \ \ \forall m \in \Ve,\ x_{1^-} = y_{1^-} = 0. \slabel{eqn:nonnegativity}
\end{subeqnarray}

In the above formulation, the objective is to minimize the expected total cost, which is equal to the total generation cost (i.e., start-up, shut-down, and fuel costs) minus the revenue, where 
{$f(\cdot)$ is quadratic as the deterministic \rred{single-}UC case and can be approximated by a piecewise linear function.}
Constraints \eqref{eqn:min-up} represent the minimum-up time {limits} for the generator. That is, if {the generator} starts up at node $i$, then it should {stay} online for all the nodes in $\He_{L}(m)$. {Similarly, constraints \eqref{eqn:min-down} represent the minimum-down time limits. If the generator shuts down at node $m$, then it should be kept offline for all the nodes in $\He_{\ell}(m)$. Constraints \eqref{eqn:u-def-1} and \eqref{eqn:u-def-2} describe {the relationship between $u$ and $y$.} Constraints \eqref{eqn:generator-bound} describe the upper and lower bounds of electricity generation amount if the generator is online at node $m$. Constraints \eqref{eqn:ramp-up} and \eqref{eqn:ramp-down} describe the ramp-up rate and ramp-down rate limits, respectively.

\subsection{An Optimality Condition}
{We first derive a property for the extreme points of conv($\mathcal{W}$), where}
 $\mathcal{W} = \{(x, y, u) \in \R^{|\Ve|} \times \B^{2|\Ve| }: \eqref{eqn:min-up} - \eqref{eqn:ramp-down}\}$, {before exploring the optimality condition.}

\begin{proposition} \label{prop:suc_ext_point}
For any extreme point $(\bar{x}, \bar{y}, \bar{u})$ of conv($\mathcal{W}$), $\bar{x}_m \in \mathcal{Q}$ for all $m \in \Ve$.
\end{proposition}
\begin{proof}
{We prove this claim by a contradiction method.}
Suppose that there exists some $m \in \Ve$ such that $\bar{x}_m \notin \mathcal{Q}$ for an extreme point $(\bar{x}, \bar{y}, \bar{u})$ of conv($\mathcal{W}$), i.e., $\bar{x}_m \in (\Clower, \Cupper) \setminus \mathcal{Q}$. In the following we construct two feasible points of conv($\mathcal{W}$) to represent $(\bar{x}, \bar{y}, \bar{u})$ in a convex combination of these two points, leading to the contradiction. Before that, we first construct a subtree of $\Ve$. If \rred{one or both of} the following conditions hold,
\begin{itemize}[nolistsep]
\item $t(m) \geq 2$ and $|x_m - x_{m^-}|=V$,
\item $t(m) \leq T-1$ and $|x_m - x_j|=V$ for some $j$ such that $m = j^-$,
\end{itemize}
then we construct a subtree of $\Ve$ that consists of nodes around node $m$ (e.g., the subtree that consists of blue nodes in Figure \ref{fig:subtree}), denoted as $\mathcal{K}(m)$, such that for $\forall n_1, n_2 \in \mathcal{K}(m)$ with $n_1^- = n_2$ or $n_2 = n_1^-$, $|x_{n_1} - x_{n_2}| = V$, and for some node $n \in \mathcal{K}(m)$ (denoted as boundary node of $\mathcal{K}(m)$), $t(n) \in \{1,T\}$ or $\exists b \in \Ve \setminus \mathcal{K}(m)$ with $n=b^-$ or $b=n^-$ such that $|x_b-x_n| \neq V$.
Otherwise, we let $\mathcal{K}(m) = \{m\}$. It is easy to observe that for any node $s \in \mathcal{K}(m)$ with $s \neq m$, there exists a unique shortest path to connect nodes $s$ and $m$, and we define the distance between nodes $s$ and $m$, 
denoted as $dist(s,m)$, as the number of edges on this unique path, i.e., the number of nodes on this unique path minus one. 
For example, in Figure \ref{fig:subtree}, $dist(s,m)=4$. We consider two points $(\bar{x}^1, \bar{y}, \bar{u})$ and $(\bar{x}^2, \bar{y}, \bar{u})$ such that $\bar{x}^1_r = \bar{x}_r + \epsilon$ for $r \in \mathcal{K}(m)$, $\bar{x}^2_r = \bar{x}_r - \epsilon$ for $r \in \mathcal{K}(m)$, and $\bar{x}^1_r = \bar{x}^2_r = \bar{x}_r$ for $r \notin \mathcal{K}(m)$, where $\epsilon$ is an arbitrarily small positive number.

\begin{figure}[!htbh]
\begin{center}
\begin{tikzpicture}[scale=0.5]
\draw[line width=0.5pt] (-2, 0) -- (0, 0);
\draw[line width=0.5pt] (-2, 0) -- (0, 3);
\draw[line width=0.5pt] (-2, 0) -- (0, -3);
\draw[line width=0.5pt, dashed] (0, 0) -- (2, 0);
\draw[line width=0.5pt, dashed] (0, 3) -- (1, 3);
\draw[line width=0.5pt, dashed] (0, -3) -- (1, -3);
\draw[line width=0.5pt] (2, 0) -- (4, 0);
\draw[line width=0.5pt] (2, 0) -- (4, 3);
\draw[line width=0.5pt] (2, 0) -- (4, -3);
\draw[line width=0.5pt, dashed] (4, 0) -- (5.5, 0);
\draw[line width=0.5pt] (4, 3) -- (6, 3);
\draw[line width=0.5pt] (4, -3) -- (6, -3);
\draw[line width=0.5pt, dashed] (6, 3) -- (8, 3);
\draw[line width=0.5pt, dashed] (6, -3) -- (8, -3);

\draw[line width=0.5pt] (8, 3) -- (10, 5);
\draw[line width=0.5pt] (8, 3) -- (10, 1);
\draw[line width=0.5pt] (8, -3) -- (10, -5);
\draw[line width=0.5pt] (8, -3) -- (10, -1);

\draw[line width=0.5pt] (10, 5)-- (12, 5);
\draw[line width=0.5pt] (10, 1)-- (12, 1);
\draw[line width=0.5pt] (10, -5)-- (12, -5);
\draw[line width=0.5pt] (10, -1)-- (12, -1);

\draw[line width=0.5pt, dashed] (12, 5)-- (14, 5);
\draw[line width=0.5pt, dashed] (12, 1)-- (14, 1);
\draw[line width=0.5pt, dashed] (12, -5)-- (14, -5);
\draw[line width=0.5pt, dashed] (12, -1)-- (14, -1);

\draw(-2, 0) circle (0.5cm);
\fill[white] (-2, 0) circle (0.46cm);
\draw (-2,0.5) node[anchor=north] {$1$};

\draw(0, 0) circle (0.5cm);
\fill[white] (0, 0) circle (0.46cm);

\draw(0, 3) circle (0.5cm);
\fill[white] (0, 3) circle (0.46cm);
\draw(0, -3) circle (0.5cm);
\fill[white] (0, -3) circle (0.46cm);    

\draw(2, 0) circle (0.5cm);
\fill[white] (2, 0) circle (0.46cm);

\draw(4, 0) circle (0.5cm);
\fill[white] (4, 0) circle (0.46cm);    

\draw(4, 3) circle (0.5cm);
\fill[white] (4, 3) circle (0.46cm);

\draw(6, 3) circle (0.5cm);
\fill[white] (6, 3) circle (0.46cm);

\draw(8, 3) circle (0.5cm);
\fill[white] (8, 3) circle (0.46cm);

\fill[blue] (4, -3) circle (0.5cm);

\fill[blue] (6, -3) circle (0.5cm);

\fill[blue] (8, -3) circle (0.5cm);

\draw(10, 5) circle (0.5cm);
\fill[white] (10, 5) circle (0.46cm);

\draw(12, 5) circle (0.5cm);
\fill[white] (12, 5) circle (0.46cm);
\draw(14, 5) circle (0.5cm);
\fill[white] (14, 5) circle (0.46cm);

\fill[blue] (10, -1) circle (0.5cm);
\draw (10,-1.5) node[anchor=north] {$m$};
\fill[blue] (12, -1) circle (0.5cm);
\draw(14, -1) circle (0.5cm);
\fill[white] (14, -1) circle (0.46cm);

\draw(10, 1) circle (0.5cm);
\fill[white] (10, 1) circle (0.46cm);
\draw(12, 1) circle (0.5cm);
\fill[white] (12, 1) circle (0.46cm);
\draw(14, 1) circle (0.5cm);
\fill[white] (14, 1) circle (0.46cm);

\fill[blue] (10, -5) circle (0.5cm);
\fill[blue] (12, -5) circle (0.5cm);
\fill[blue] (14, -5) circle (0.5cm);
\draw (14,-6.5) node[anchor=south] {$s$};


\end{tikzpicture}
\caption{Subtree representation} \label{fig:subtree}
\end{center}
\end{figure}

Now we show these two points constructed are feasible for conv($\mathcal{W}$) by considering the following three possible cases.
\begin{enumerate}[(1)]
\item If there exists some boundary node $n$ of $\mathcal{K}(m)$ such that $\bar{x}_{b} = 0$ with $b=n^-$ or $n=b^-$, then we have $\Clower < \bar{x}_{n} < \Vupper$. Otherwise, 1) If $\bar{x}_{n} = \Clower$, it follows that $\bar{x}_m = \Clower + kV$ for some $k \in [0, dist(n,m)]_{\Z}$ by definition. Since $\bar{x}_m \notin \mathcal{Q}$, it further follows that $k \geq \alpha_1+1 = \min \{T, \lfloor (\Cupper-\Clower)/V \rfloor \}+1$, which contradicts to the fact that $k \leq T$ and $\bar{x}_t < \Cupper$. 2) If $\bar{x}_{n} = \Vupper$, it follows that $\bar{x}_m = \Vupper + kV$ for some $k \in [0, dist(n,m)]_{\Z}$ by definition. Since $\bar{x}_m \notin \mathcal{Q}$, it further follows that $k \geq \alpha_2+1 = \min \{T, \lfloor (\Cupper-\Vupper)/V \rfloor \}+1$, which contradicts to the fact that $k \leq T$ and $\bar{x}_m < \Cupper$. Therefore, it is feasible to increase or decrease $\bar{x}_{n}$ by $\epsilon$.
\item If there exists some boundary node $n$ of $\mathcal{K}(m)$ such that $\bar{x}_{b} > 0$ with $b=n^-$ or $n=b^-$, then we have $\Clower < \bar{x}_{n} < \Cupper$. Otherwise, 1) If $\bar{x}_{n} = \Clower$, we can similarly show the contradiction as above. 2) If $\bar{x}_{n} = \Cupper$, it follows that $\bar{x}_m = \Cupper - kV$ for some $k \in [0, dist(n,m)]_{\Z}$ by definition. Since $\bar{x}_m \notin \mathcal{Q}$, it further follows that $k \geq \alpha_1+1 = \min \{T, \lfloor (\Cupper-\Clower)/V \rfloor \}+1$, which contradicts to the fact that $k \leq T$ and $\bar{x}_t > \Clower$. Therefore, it is feasible to increase or decrease $\bar{x}_{n}$ by $\epsilon$ since $|\bar{x}_{n} - \bar{x}_{b}| < V$ by definition.
\item If there does not exist node {$b \in \Ve$} with $b=n^-$ or $n=b^-$ for some boundary node $n$ of $\mathcal{K}(m)$, i.e., $t(n) \in \{1, T\}$, then similarly we can follow the arguments above to show that $\Clower < \bar{x}_{n} < \Cupper$. It follows that it is feasible to increase or decrease $\bar{x}_{n}$ by $\epsilon$.
\end{enumerate}

In summary, we show that in all cases it is feasible to increase or decrease $\bar{x}_{n}$ by $\epsilon$ for each boundary node $n$ of $\mathcal{K}(m)$ and thus feasible to increase or decrease $\bar{x}_r$ by $\epsilon$ for all $r \in \mathcal{K}(m)$. It follows that both $(\bar{x}^1, \bar{y}, \bar{u})$ and $(\bar{x}^2, \bar{y}, \bar{u})$ are feasible points of conv($\mathcal{W}$), and $(\bar{x}, \bar{y}, \bar{u}) = \frac{1}{2}(\bar{x}^1, \bar{y}, \bar{u}) + \frac{1}{2}(\bar{x}^2, \bar{y}, \bar{u})$. Therefore, $(\bar{x}, \bar{y}, \bar{u})$ is not an extreme point of conv($\mathcal{W}$), which is a contradiction.
\end{proof}

Now, we begin to characterize the optimality condition for {the stochastic \rred{single-}UC problem} \eqref{eqn:MSS}. Similar to the deterministic {\rred{single-}UC} problem, 
 {$f(\cdot)$} is approximated by a $K-$piece piecewise linear function.
 {Accordingly, formulation} \eqref{eqn:MSS} can be reformulated as
\begin{eqnarray}
&\min  & \sum_{m \in \Ve } p_m \bigg( \overline{U}u_m + \underline{U}(y_{m^-} - y_m + u_m) +\varphi_m \bigg) \nonumber \\
&\mbox{s.t.}& (x,y,u) \in \mathcal{W}, \nonumber \\
&& \varphi_m \geq \mu_k^m x_m + \nu_k y_m, \ \forall k \in [1,K]_{\Z}, \forall m \in \Ve. \label{eqn:suc_piecewise_cons}
\end{eqnarray}


\begin{proposition} \label{prop:suc_opt_sol}
Problem \eqref{eqn:MSS} has at least one optimal solution $(\bar{x},\bar{y},\bar{u})$ with $\bar{x}_m \in \mathcal{Q}_d$ for all $m \in \Ve$.
\end{proposition}
\begin{proof}
{We prove this claim by a contradiction method.} Suppose that there exists some $m \in \Ve$ such that $\bar{x}_m \notin \mathcal{Q}_d$ for the optimal solution $(\bar{x}, \bar{y}, \bar{u})$ of Problem \eqref{eqn:MSS}, i.e., $\bar{x}_m \in (\Clower, \Cupper) \setminus \mathcal{Q}_d$, with the optimal value $\bar{z} = \sum_{m \in \Ve } \big(  \overline{U} \bar{u}_m + \underline{U}(\bar{y}_{m^-} - \bar{y}_m + \bar{u}_m) +  \bar{\varphi}_m \big)$. In the following we construct a feasible solution to obtain a better objective value. Following the proof in Proposition \ref{prop:suc_ext_point}, we construct a subtree $\mathcal{K}(m)$ as shown in Figure \ref{fig:subtree} and we can always construct two feasible points of $\mathcal{W}$, $(\bar{x}^1, \bar{y}, \bar{u})$ and $(\bar{x}^2, \bar{y}, \bar{u})$, such that $\bar{x}^1_r = \bar{x}_r + \epsilon$ for $r \in \mathcal{K}(m)$, $\bar{x}^2_r = \bar{x}_r - \epsilon$ for $r \in \mathcal{K}(m)$, and $\bar{x}^1_r = \bar{x}^2_r = \bar{x}_r$ for $r \notin \mathcal{K}(m)$, where $\epsilon$ is an arbitrarily small positive number. Similar to the proof in Proposition \ref{prop:duc_opt_sol}, we can increase or decrease $\bar{x}_r$ by $\epsilon$ for $r \in \mathcal{K}(m)$ to decrease the optimal value by at least $| \sum_{r \in \mathcal{K}(m)} \mu_{k_r}^r | \epsilon$ and the resulting solution, $(\bar{x}^1, \bar{y}, \bar{u})$ or $(\bar{x}^2, \bar{y}, \bar{u})$, is feasible for both $\mathcal{W}$ and constraints \eqref{eqn:suc_piecewise_cons}, where there is at most one piece of linear function of constraints \eqref{eqn:suc_piecewise_cons}, e.g., piece $k_r$, is tight for each node $r \in \mathcal{K}(m)$. Therefore we obtain the contradiction.
\end{proof}

In other words, in order to find an optimal solution to the multistage stochastic \rred{single-}UC problem, we only need to consider the feasible solutions $(x,y,u)$ where $x_m \in \mathcal{Q}_d$ for all $m \in \Ve$.
%

\subsection{An $\Oe(N)$ Time Dynamic Programming Algorithm}
We can explore the backward induction dynamic programming framework by
reusing the state space and the state transition graph we defined in Section \ref{subsec:duc-ot-dp}.
Let $ F_m (i)$ represent the optimal value function for node $ m \in \Ve $ in the scenario tree {in which the generator state of node $ m^- $ is state $i$}.
Based on Proposition \ref{prop:suc_opt_sol}, an optimal decision for {the} current state lies in set $ S(i) $. {Accordingly}, the Bellman equations can be formulated as follows:
\begin{eqnarray} \label{sto: bellman}
	&& F_m (i) = \min_{j\in S(i)} \bigg\{ E_{mij}  + \sum_{n\in \C_*(m)} F_{n}(j) \bigg\}, \  \forall i \in  \mathcal{S}, \forall m \in \Ve,  \slabel{sto: bellman2}
\end{eqnarray}
where node $ m $, replacing time $ t $ in \eqref{model:duc_lp_bell}, represents {a scenario tree} node, and $ E_{mij} = p_m \big( \overline{U}u(j) + \underline{U}(y(i) - y(j) +u(j)) + f_m(x(j), y(j)) \big) $ to describe the total generation cost minus the revenue at node $ m \in \Ve $. 
Note {here that} when $ m $ is a leaf node, $ C_*(m) = \emptyset $ and {accordingly} $\sum_{n\in \C_*(m)} F_{n}(j) = 0$. The objective is to find out the value of {$ F_1(0) $} where the {label ``$1$''} in the subscript represents the root node in the scenario tree while the {label ``$0$''} in the bracket represents the source state in the state space.

\begin{proposition} \label{prop:suc_orderN}
The stochastic \rred{single-}UC problem \eqref{eqn:MSS} can be solved in $ \Oe (N) $ time, \rred{or more specifically, in $ \Oe ({|\mathcal{Q}_d|}({|\mathcal{Q}_d|} L + \ell)N ) $} time.
\end{proposition}
\begin{proof}
In order to {obtain} the optimal objective value and optimal solution for the scenario-tree based multistage stochastic {\rred{single-}UC} problem, we need to calculate $ F_m (i) $ for all possible $m$ and $i$ and record the optimal candidates for them. To calculate the value of each optimal value function $ F_m (i) $ in the Bellman equations \eqref{sto: bellman}, we search among the candidate solution $ j\in S(i) $ and this step takes $ \Oe({|\mathcal{Q}_d|} |\C_*(m)|) $ time. Since there are in total $ {|\mathcal{Q}_d|} L + \ell $ number of nodes in the state graph, the computational time at each node $ m $ is $ \Oe ({|\mathcal{Q}_d|} ({|\mathcal{Q}_d|} L + \ell) |\C_*(m)|). $ Thus, the total time to calculate $  F_1(0) $ is $ \Oe ({|\mathcal{Q}_d|}({|\mathcal{Q}_d|} L + \ell)N ) $, where $ N$ represents the total number of nodes in the scenario tree. 
\end{proof}

Similar to Remark \ref{rmk:duc-su}, when start-up profile is considered, the conclusion in Proposition \ref{prop:suc_orderN} still holds.

\subsection{An Extended Formulation for Stochastic \rred{Single-}UC with Piecewise Linear Cost Function} \label{exsuc}
Following the same approach {as that in Section \ref{exuc},} we develop an extended formulation in linear program form for multistage stochastic {\rred{single-}UC} problem, which is proved to provide integral solutions. By incorporating the Bellman equations \eqref{sto: bellman}  as constraints, we can derive the following primal linear program similarly:
\begin{subeqnarray} \label{eqn:sto-bell}
	&\max  & F_1(0) \\
	&\mbox{s.t.}  & F_m(i) \leq E_{mij} + \sum_{n\in \C_*(m)} F_{n}(j), \ \forall i \in \mathcal{S}, j \in S(i), \forall m \in \Ve, \slabel{eqn:sto_p}
\end{subeqnarray}
where $ E_{mij} $ are parameters defined under equations \eqref{sto: bellman2} and $ F_m(i) $ are decision variables.

Similar to formulation \eqref{model:duc_lp_bell}, here we also resort to the dual formulation and then provide an extended linear {program}, which we {show} can provide solution to the {stochastic} \rred{single-}UC problem directly.
The dual formulation can be derived accordingly as follows:
\begin{subeqnarray} \label{duc_d}
	&\min  & \sum_{m \in \Ve, i \in \mathcal{S}, j \in S(i)} E_{mij}w_{mij} \\
	&\mbox{s.t.}  & \sum_{j \in S(1)} w_{11j} = 1, \slabel{sto:dual1}\\
	&& \sum_{j \in S(i)} w_{1ij} = 0, \ \forall i \in \mathcal{S} \setminus \mathcal{S}_*,\slabel{sto:dual2} \\	
	&& \sum_{j \in S(i)} w_{mij} - \sum_{k \in P(i)} w_{m^-,ki} = 0, \ \forall  i \in \mathcal{S}, \forall m \in \Ve \setminus \Ve_*, \slabel{sto:dual3} \\	
	&& w_{mij} \geq 0, \ \forall i \in \mathcal{S}, j \in S(i), \forall  m \in \Ve, \slabel{sto:dual4}
\end{subeqnarray}
where $  w_{mij} $ are dual variables corresponding to each constraint in the primal linear program \eqref{eqn:sto-bell}.

In the following, we {first} prove in a similar way to show that the polytope \eqref{sto:dual1} -- \eqref{sto:dual4} is an integral polytope.
\begin{lemma} \label{lemmasto:binary solution}
	The extreme points of the polytope \eqref{sto:dual1} -- \eqref{sto:dual4} are binary.
\end{lemma}
\begin{proof}
	The proof is similar to that in Lemma \ref{lemma:binary solution}. $\forall \ E_{mij} \in {\R}$ , we prove there exists at least one optimal solution to the dual program \eqref{duc_d} that is binary. By solving the dynamic program for the stochastic \rred{single-}UC model, we can obtain an optimal decision of $ (x,y,u) $. We then construct $ \hat{w} $, which {is} binary, to represent the optimal solution. We let $ \hat{w}_{mij} = 1$ if on the {scenario tree} node $ m $ the optimal dynamic programming solution {shows the} decision from state $ i $ to state $ j $ in the state space. Otherwise, we let $ \hat{w}_{mij} = 0$.
	
	Similarly we can easily show that the constructed $ \hat{w} $ is feasible to the dual program by verifying that it satisfies constraints \eqref{sto:dual1} -- \eqref{sto:dual4}. Also, the optimality can be proved following the same approach in Lemma \ref{lemma:binary solution}. As $ \hat{w} $ is binary and optimal to the dual program {for} {any possible} cost coefficient $ E$, we have proved our claim.
\end{proof}

Next, we show that an optimal solution to the stochastic \rred{single-}UC problem \eqref{eqn:MSS} can be obtained in terms of the dual variables.
\begin{proposition} \label{thmsto: opt solution}
	If $ w^* $ is an optimal solution to {the} dual program \eqref{duc_d}, then
	\begin{eqnarray} \label{eqnsto:opt solution}
		&& x_m^* = \sum_{i \in \mathcal{S}, j \in S(i)} x(j) w_{mij}^*, y_m^* = \sum_{i \in \mathcal{S}, j \in S(i)} y(j) w_{mij}^*,  u_{{m}}^* = \sum_{i \in \mathcal{S}, j \in S(i)} u(j) w_{mij}^*, \  m \in \Ve, \label{eqnsto:opt solution1} 
	\end{eqnarray}
	is an optimal solution to the stochastic \rred{single-}UC problem \eqref{eqn:MSS}.
\end{proposition}
\begin{proof}
	The proof is similar to that in Proposition \ref{thm: opt solution}. The feasibility is satisfied because, from the proof of Proposition \ref{thm: opt solution}, for each scenario we can consider the multistage decision as a path in the directed graph and thus it satisfies all the physical and logical constraints. The optimality can be proved by replacing the decision $(x^*, y^*, u^*)$ with the corresponding expressions in equations \eqref{eqnsto:opt solution} in the objective function and verify that 
	\begin{eqnarray} 
		\sum_{ m \in \Ve } p_m \bigg(  \overline{U}u_m^* + \underline{U}(y_{m^-}^* - y_m^* + u_m^*) +  f_m(x_m^*,y_m^*) \bigg) =   \sum_{m \in \Ve, i \in \mathcal{S}, j \in S(i)} E_{mij}w_{mij}^*. \nonumber
	\end{eqnarray}
	Hence,  $ (x^*,y^*,u^*) $ is the optimal solution to the stochastic \rred{single-}UC problem.
\end{proof}

Now we are ready to establish the extended formulations for {the} stochastic {\rred{single-}UC} problem. We replace constraints \eqref{eqn:min-up} -- \eqref{eqn:nonnegativity} with constraints \eqref{sto:dual1} -- \eqref{sto:dual4} and add equations \eqref{eqnsto:opt solution} to represent the {relationship} between original decisions and the dual decision variables.
\begin{theorem}
	The extended formulation of the stochastic \rred{single-}UC problem \eqref{eqn:MSS} can be written as follows:
	\begin{subeqnarray} 
		&\max  & \sum_{m \in \Ve } p_m \bigg(  \overline{U}u_m + \underline{U}(y_{m^-} - y_m + u_m) + \rblue{\varphi_m } \bigg) \nonumber \\
		&\mbox{s.t.}  &  x_m = \sum_{i \in \mathcal{S}, j \in S(i)} x(j) w_{mij}, \ y_m = \sum_{i \in \mathcal{S}, j \in S(i)} y(j) w_{mij}, \nonumber \\
		&& u_m = \sum_{i \in \mathcal{S}, j \in S(i)} u(j) w_{mij}, \ \forall m \in \Ve, \ \rblue{\eqref{eqn:suc_piecewise_cons},} \ \eqref{sto:dual1} - \eqref{sto:dual4}, \nonumber
	\end{subeqnarray}
	and if $ (x^*,y^*,u^*,w^*) $ is an optimal solution to the extended formulation, then  $ (x^*,y^*,u^*) $ is an optimal solution to the stochastic \rred{single-}UC problem.
\end{theorem}
\begin{proof}
	The proof is similar to that in Theorem \ref{thm:extend}. {The details are omitted here.}
\end{proof}



\section{Conclusion} \label{con}
In this paper, efficient dynamic programming algorithms and linear program reformulations {were} proposed to solve the deterministic and stochastic {\rred{single-}UC} problems. We started with deriving a more efficient dynamic programming algorithm to solve the deterministic {\rred{single-}UC} problem with {a} general convex cost function. Our proposed algorithm refines a previous work by enhancing the computational time {solving \rred{single-}UC} from $ \Oe(T^3) $ time to $ \Oe(T^2) $ time when the economic dispatch problems are solved in advance.
Motivated by this, we obtained an integral polytope to describe the integer feasible region of the deterministic \rred{single-}UC problem. 
Meanwhile, we derive 
an efficient extended reformulation in a higher dimensional space that can provide integral solutions for the deterministic \rred{single-}UC problem.
In addition, for the most common piecewise linear cost {objective} function {case}, by exploiting the optimality condition for the {deterministic \rred{single-}UC problem}, we proposed a more efficient dynamic programming algorithm that {runs in} $ \Oe(T) $ time and furthermore, 
our study was {extended} to solve the stochastic {\rred{single-}UC} problem in $ \Oe(N) $ time {by also {deriving the} corresponding optimality condition}. Extended formulations {were} {further} derived for {both deterministic and stochastic \rred{single-}UC problems} and integral solutions for {both of them} {were} provided.

Our studies provide efficient polynomial time algorithms for a class of \rred{single-}UC problems, especially linear time for certain cases. Furthermore, we provide one of the first studies on deriving extended formulations for various \rred{single-}UC problems based on efficient dynamic programming algorithms. Besides solving single-generator self-scheduling/bidding problems, our polynomial time algorithms and/or extended formulations could potentially help speed up the MILP and Lagrangian Relaxation approaches to solve the \rred{multi-UC} problems efficiently.

\section*{Acknowledgments}
The authors thank the editor and the anonymous referees for their sincere suggestions on improving the quality of this paper. The work of K. Pan was partially supported by Hong Kong Polytechnic University under grants 1-ZE73 and G-UABE.


\baselineskip=11pt
\bibliographystyle{plain}
\bibliography{duc}

\newpage
\begin{appendices}
\baselineskip=22pt
\setcounter{proposition}{0}
\setcounter{lemma}{0}
\setcounter{theorem}{0}

\end{appendices}

\end{document}